\documentclass[12pt]{article}
\usepackage{amsmath,amsfonts,amssymb,amsthm,bbm,tikz-cd,
  etoolbox,xparse,hyperref,lineno,subcaption,rotating}
\usepackage[linesnumbered,ruled,vlined]{algorithm2e}
\usepackage[active]{srcltx}
\usepackage[utf8]{inputenc}

\DeclareMathOperator{\argmin}{argmin}
\DeclareMathOperator{\dist}{d}
\DeclareMathOperator{\rd}{rd}
\DeclareMathOperator{\bd}{bd}
\DeclareMathOperator{\tr}{Tr}

\tikzset{
  @go line/.code args={#1#2#3!}{
    \edef\myrest{#3}
    \edef\myletter{#2}
    \begin{pgfonlayer}{background layer}
    \draw (#1,-\mycount) +(0,-.5) -- +(0,.5) +(.5,0) -- +(-.5,0);
    \end{pgfonlayer}    
    \ifdefstring{\myletter}{1}{
      \path[draw,fill=white] (#1,-\mycount) circle (.4);
    }{
      \ifdefstring{\myletter}{0}{
        \path[draw,fill=black] (#1,-\mycount) circle (.4);
      }{
        \ifdefstring{\myletter}{B}{
                \path[draw,fill=black] (#1,-\mycount) circle (.4);
          \draw [white,fill=black] (#1-1,-\mycount-1) rectangle (#1+1,-\mycount+1);
        }{
          \ifdefstring{\myletter}{b}{
            \path[draw,fill=black] (#1,-\mycount) circle (.4);
            \draw [white,fill=black]
              (#1-0.5,-\mycount-0.5) rectangle (#1+0.5,-\mycount+0.5);
          }{}
        }
      }
    }
    \ifdefstring{\myrest}{}{}{
      \pgfmathsetmacro\mynext{int(#1+1)}
      \pgfkeysalso{@go line={\mynext}#3!}
    }
  },
  go line/.style={@go line={0}#1!},
}

\NewDocumentCommand\showgoboard{m}{
  \begin{tikzpicture}[x=\mygounit,y=\mygounit]
    \pgfdeclarelayer{background layer}
    \pgfsetlayers{background layer,main}

    \edef\myconf{#1}
    \foreach [count=\mycount] \myline in \myconf {
      \tikzset{go line/.expanded=\myline}
    }
  \end{tikzpicture}
}

\def\mygounit{2.5mm}

\begin{document}
\newtheorem{abstractclass}{abstractclass}
\newtheorem{definition}[abstractclass]{Definition}
\newtheorem{lemma}[abstractclass]{Lemma}
\newtheorem{proposition}[abstractclass]{Proposition}
\newtheorem{theorem}[abstractclass]{Theorem}
\newtheorem{example}[abstractclass]{Example}

\newcommand{\R}{\mathbbm{R}}
\newcommand{\N}{\mathbbm{N}}
\newcommand{\Z}{\mathbbm{Z}}

\title{Restriction and interpolation operators for digital images and their
boundaries}
\author{Janosch Rieger%
\footnote{janosch.rieger@monash.edu, Monash University, Australia (corresponding author)} 
\and 
Kyria Wawryk%
\footnote{kyria.wawryk@monash.edu, Monash University, Australia}}

\maketitle

\begin{abstract}
The aim of this paper is to provide a coherent framework for
transforming boundary pairs of digital images from one resolution to another
without knowledge of the full images.
It is intended to facilitate the simultaneous usage of multiresolution
processing and boundary reduction, primarily for algorithms 
in computational dynamics and computational control theory.
\end{abstract}

\noindent\textbf{Keywords:} digital images, multiresolution processing,
boundary reduction, computational dynamics, computational control theory

\medskip 

\noindent\textbf{Mathematics Subject Classification:} 68U05, 37M22, 93B03, 65D18

\section{Introduction}

Digital images are not only 
of interest in digital geometry \cite{Klette}.
They also occur, often under the name \emph{box covering} \cite{complicated} 
or \emph{grid covering} \cite{Mischaikow:0}, in completely unrelated disciplines 
like computational dynamics, computational control theory and rigorous computing.
The data structure underlying the software package GAIO for the computation 
of invariant objects \cite{Dellnitz,Rieger:GAIO},
the viability kernel algorithm \cite{Saint-Pierre},
and algorithms for the computation of reachable sets \cite{Rieger}, 
is indeed a digital image.

\medskip

Multiresolution processing and boundary reduction
are techniques for speeding up algorithms based on digital image 
representations:
Algorithms derived from GAIO, in the field of computational dynamics
\cite{Mischaikow:1,Junge,Mischaikow:2,Rasmussen}
and beyond \cite{Hestermeyer,Sertl},
as well as many algorithms in rigorous computing \cite{Jaulin},
all start with a very crude box covering of the region of interest, 
and alternate between an elimination step in which boxes are deleted 
from the cover, and a refinement step in which the remaining boxes are subdivided.
This approach allows for large gains 
at low cost in the early stages of the algorithm.
On the other hand, the boundary tracking algorithm \cite{Rieger} reduces the complexity 
of reachable set computation by storing and manipulating a \emph{boundary pair}, 
i.e.\ the boundary of a digital image together with one adjacent layer in the complement, 
instead of working with the full digital image.

\medskip

The aim of this paper is to provide a coherent framework for
transforming the boundary pairs of digital images from one resolution to another
without knowledge of the full image.
It is intended to facilitate the simultaneous usage of multiresolution
processing and boundary reduction.
We contribute mainly the following insights:
\begin{itemize}
\item [i)] We explore the structure of digital images in terms of inner and outer layers. 
This is related to the so-called distance transform \cite[Section 3.4.2]{Klette},
but we are interested in the geometry of the layers rather than their computation,
which we aim to avoid.
\item [ii)] We characterize the space of all boundary pairs of digital images.
\item [iii)] We identify a restriction operator $R$ that projects a digital image
living on a fine grid onto a coarser grid, 
and an interpolation operator $I$ that refines digital images, 
such that both operators satisfy a number of desirable properties, individually and as a pair.
\item [iv)] We lift the operators $R$ and $I$ to mappings $\partial R$ and $\partial I$
between the spaces of boundary pairs corresponding to the coarse and fine grids 
from iii), and provide implementations of
these lifts that neither require knowledge of the full images nor compute this information
implicitly.
\end{itemize}
A more detailed outline of the paper is postponed to Section \ref{sec:outline},
when the definitions and the terminology have been fixed.

\section{Definitions and outline}

We collect most definitions from this paper in Section \ref{sec:definitions}.
They serve as the vocabulary for presenting a technical outline 
in Section \ref{sec:outline}.

\subsection{Definitions}\label{sec:definitions}

We denote $\N_0=\{0,1,2,\ldots\}$, $\N_1=\{1,2,3,\ldots\}$ and $\N_2=\{2,3,4,\ldots\}$.
When $J\subset\R$ is an interval, and when it is clear that $k\in\N$ or $k\in\Z$, 
we will write $k\in J$ instead of $k\in J\cap\N$ or $k\in J\cap\Z$.
For every $x\in\R$, we denote
\[
\lfloor x\rfloor:=\max\{k\in\Z:k\le x\}
\quad\text{and}\quad
\lceil x\rceil:=\min\{k\in\Z:x\le k\}.
\]

Let $m\in\N_1$.
We equip $\R^m$ with the norm $\|x\|_\infty:=\max_{j\in\{i,\ldots,m\}}|x_j|$.
For any $\emptyset\neq M,M'\subset\R^m$ 
we consider the Hausdorff semi-distance and the full Hausdorff distance
\begin{align*}
\dist(M,M')&:=\sup_{x\in M}\inf_{x'\in M'}\|x-x'\|_\infty,\\
\dist_H(M,M')&:=\max\{\dist(M,M'),\dist(M',M)\},
\end{align*}
with the additional convention that $\dist(M,\emptyset)=\infty$ 
for all $\emptyset\neq M\subset\R^m$.
For all $M\subset\R^m$ and $\delta>0$, we will denote
$B_\delta(M):=\{x\in\R^m:\dist(x,M)\le\delta\}$.

\medskip

We introduce grids, 
their subsets, and boundary layers of these sets.
Definitions \eqref{def:partial:1} through \eqref{def:partial:4} are as in \cite{Rieger}.
For an illustration see Figure \ref{fig:anatomy}.

\begin{definition}\label{eq:bdrylayers}
For every $\rho>0$, consider the grid
\[\Delta_\rho:=\rho\Z^m\]
as well as the collections
\[S_\rho^-:=\{M\subset\Delta_\rho: \emptyset\neq M\neq\Delta_\rho\},\quad
S_\rho:=\{M\subset\Delta_\rho: M\neq\emptyset\},\quad
S_\rho^+:=2^{\Delta_\rho}.\]
For every $M\in S_\rho^+$ we define the (possibly empty) sets
\begin{align}
\partial_\rho^0M
&:=\{x\in M:\exists\,z\in M^c\cap\Delta_\rho\ \text{with}\ \|x-z\|_\infty=\rho\},
\label{def:partial:1}\\
\partial_\rho^kM&:=\{z\in M^c\cap\Delta_\rho:
\dist(z,\partial_\rho^0M)=k\rho\},\quad k\in\N_1,
\label{def:partial:3}\\
\partial_\rho^{-k}M&:=\{x\in M:
\dist(x,\partial_\rho^0M)=k\rho\},\quad k\in\N_1.
\label{def:partial:4}
\end{align}
\end{definition}

\begin{figure*}[t]
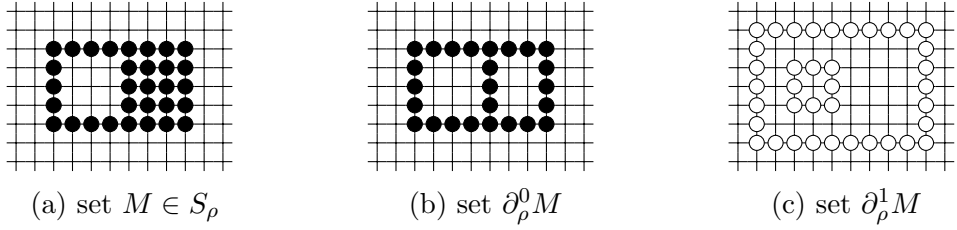

\centering
\begin{subfigure}[t]{0.3\textwidth}
\centering
\showgoboard{
------------,
------------,
--00000000--,
--0---0000--,
--0---0000--,
--0---0000--,
--00000000--,
------------,
------------}
\caption{set $M\in S_\rho$}
\end{subfigure}
\hfill
\begin{subfigure}[t]{0.3\textwidth}
\centering
\showgoboard{
------------,
------------,
--00000000--,
--0---0--0--,
--0---0--0--,
--0---0--0--,
--00000000--,
------------,
------------}
\caption{set $\partial^0_\rho M$}
\end{subfigure}
\hfill
\begin{subfigure}[t]{0.3\textwidth}
\centering
\showgoboard{
------------,
-1111111111-,
-1--------1-,
-1-111----1-,
-1-1-1----1-,
-1-111----1-,
-1--------1-,
-1111111111-,
------------}
\caption{set $\partial^1_\rho M$}
\end{subfigure}
\caption{Anatomy of a digital image from Definition \ref{eq:bdrylayers}.
\label{fig:anatomy}}
\end{figure*}

Paths help describe the topology of sets in $S_\rho$.
It is easy to check that their concatenation is well-defined.

\begin{definition}\label{def:concatenate}
Let $\rho>0$, and let $x,z\in\Delta_\rho$.
A finite sequence 
\[p=(\xi_0,\xi_1,\ldots,\xi_k)\in\Delta_\rho^{k+1}\]
with $k\in\N_0$ satisfying $\xi_0=x$, $\xi_k=z$ and 
\[\|\xi_\ell-\xi_{\ell-1}\|_\infty\le\rho\quad\forall\,\ell\in(0,k]\]
is called a path from $x$ to $z$ of length $L(p):=k$.
We write $p(\ell):=\xi_\ell$ for $\ell\in[0,k]$, and we denote the space of all paths 
from $x$ to $z$ by $P_\rho(x,z)$.

For any $x,y,z\in\Delta_\rho$, $p_0\in P_\rho(x,y)$ and $p_1\in P_\rho(y,z)$ we define 
the concatenation $p_1\circ p_0\in P_\rho(x,z)$ by
\begin{equation}\label{concatenation}
(p_1\circ p_0)(\ell):=\begin{cases}
p_0(\ell),&\ell\in[0,L(p_0)],\\
p_1(\ell-L(p_0)),&\ell\in(L(p_0),L(p_0)+L(p_1)].
\end{cases}
\end{equation}
\end{definition}

The space $\bd_\rho$ is a central object of this paper.
Much of the content of Sections \ref{sec:geometry} and \ref{sec:boundary:pairs} 
will be used in showing that it is the collection of all
pairs $(\partial_\rho^0M,\partial_\rho^1M)$ of sets $M\in S_\rho$.
For the significance of axiom \eqref{bdry:axiom:4} see Figure \ref{fig:axiom:failure}.

\begin{figure*}[t]
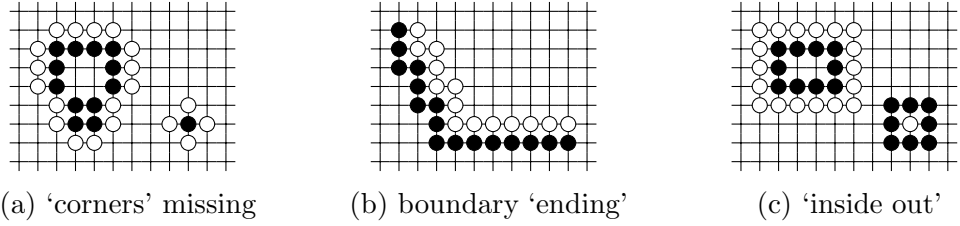

\centering
\begin{subfigure}[t]{0.3\textwidth}
\centering
\showgoboard{
------------,
--1111------,
-100001-----,
-10--01-----,
-10--01-----,
--10 01---1--,
--1001--101-,
---11----1--,
------------}
\caption{`corners' missing}
\end{subfigure}
\hfill
\begin{subfigure}[t]{0.3\textwidth}
\centering
\showgoboard{
------------,
-01---------,
-011--------,
-001--------,
--011-------,
--001-------,
---01111111-,
---00000000-,
------------}
\caption{boundary `ending'}
\end{subfigure}
\hfill
\begin{subfigure}[t]{0.3\textwidth}
\centering
\showgoboard{
------------,
-111111-----,
-100001-----,
-10--01-----,
-100001-----,
-111111-000-,
--------010-,
--------000-,
------------}
\caption{`inside out'}
\end{subfigure}
\caption{Examples of pairs $(D_0,D_1)\in S_\rho^+\times S_\rho^+$ satisfying 
axioms \eqref{bdry:axiom:0} through \eqref{bdry:axiom:3}, but not axiom 
\eqref{bdry:axiom:4} from Definition \ref{def:boundary:pairs}.\label{fig:axiom:failure}}
\end{figure*}

\begin{definition}\label{def:boundary:pairs}
Let $\rho>0$. 
Then $\bd_\rho^-\subset S_\rho^+\times S_\rho^+$ 
is the space of all pairs $(D_0,D_1)$ satisfying
\begin{align}
&D_0\neq\emptyset\neq D_1,
\label{bdry:axiom:0}\\
&D_0\cap D_1=\emptyset,
\label{bdry:axiom:1}\\
&(\forall\,x\in D_0)(\exists z\in D_1):\ \|x-z\|_\infty=\rho,
\label{bdry:axiom:2}\\
&(\forall\,z\in D_1)(\exists x\in D_0):\ \|x-z\|_\infty=\rho,
\label{bdry:axiom:3}\\
&\begin{aligned}
&(\forall\,x\in D_0)
(\forall\,z\in D_1)
(\forall\,p\in P_\rho(x,z)):\\
&\phantom{(\forall\,x\in D_0)}
(L(p)>1)\Rightarrow(\exists\,\ell\in(0,L(p))):\ p(\ell)\in D_0\cup D_1),
\end{aligned}\label{bdry:axiom:4}
\end{align}
and $\bd_\rho:=\bd_\rho^-\cup\{(\emptyset,\emptyset)\}$.
\end{definition}

We also introduce collections of sets that will allow to formalize desirable
properties of restriction and interpolation operators.

\begin{definition}\label{def:collections}
Let $\rho>0$ and $\rho'>0$.
Then we call
\[C_\rho:=\{M\in S_\rho:(\forall x,x'\in M)(\exists\,p\in P_\rho(x,x'))(\forall\,\ell\in[0,L(p)]):
(p(\ell)\in M)\}\]
the collection of all connected sets in $S_\rho$, 
and for every $M\in S_\rho$, we call
\begin{align}
A_{\rho'}(M)&:=\argmin\{\dist_H(M',M):M'\in S_{\rho'}\},\label{def:A}\\
V_{\rho}^{\rho'}(M)
&:=\{M'\in S_{\rho'}:B_{\rho/2}(M)\subset B_{\rho'/2}(M')\}\label{def:V}
\end{align}
the collection of all best approximations to $M$ in $S_{\rho'}$
and the collection of all Voronoi covers of $M$ in $S_{\rho'}$, respectively.
\end{definition}

\begin{figure*}[b]
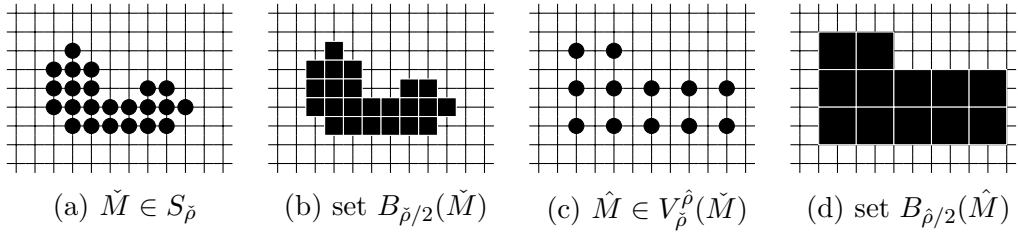

\centering
\begin{subfigure}[t]{0.24\textwidth}
\centering
\showgoboard{
------------,
------------,
---0--------,
--000-------,
--000--00---,
--00000000--,
---000000---,
------------,
------------}
\caption{$\check M\in S_{\check\rho}$}
\end{subfigure}
\hfill
\begin{subfigure}[t]{0.24\textwidth}
\centering
\showgoboard{
------------,
------------,
---b--------,
--bbb-------,
--bbb--bb---,
--bbbbbbbb--,
---bbbbbb---,
------------,
------------}
\caption{set $B_{\check\rho/2}(\check M)$}
\end{subfigure}
\hfill
\begin{subfigure}[t]{0.24\textwidth}
\centering
\showgoboard{
------------,
------------,
--0-0-------,
------------,
--0-0-0-0-0-,
------------,
--0-0-0-0-0-,
------------,
------------}
\caption{$\hat M\in V_{\check\rho}^{\hat\rho}(\check M)$}
\end{subfigure}
\hfill
\begin{subfigure}[t]{0.24\textwidth}
\centering
\showgoboard{
------------,
------------,
--B-B-------,
------------,
--B-B-B-B-B-,
------------,
--B-B-B-B-B-,
------------,
------------}
\caption{set $B_{\hat\rho/2}(\hat M)$}
\end{subfigure}
\caption{Illustration of covering property encoded in \eqref{def:V} from Definition 
\ref{def:collections} with $\rho=\check\rho$ and $\rho'=\hat\rho=2\check\rho$:
Set from (d) covers set from (b).
White lines in (b) and (d) are visual aids only.
\label{fig:Voronoi}}
\end{figure*}

The collection $V_{\rho}^{\rho'}(M)$ consists of all $M'\in S_{\rho'}$ 
such that the Voronoi cells generated by $M'$ in the Voronoi diagram generated by 
$\Delta_{\rho'}$ cover the Voronoi cells generated by $M$ in the Voronoi 
diagram generated by $\Delta_\rho$.
For an illustration see Figure \ref{fig:Voronoi}.

\subsection{Outline}\label{sec:outline}

\begin{figure}[b]
\[\begin{tikzcd}
\bd_{\hat\rho}
  \arrow[dashed]{r}{\partial I} 
&\bd_{\check\rho}
  \arrow[dashed]{r}{\partial R}
&\bd_{\hat\rho}\\ 
S_{\hat\rho} \arrow{u}{\tr_{\hat\rho}} 
  \arrow{r}{I}   
&S_{\check\rho}
  \arrow{u}{\tr_{\check\rho} }
  \arrow{r}{R} 
&S_{\hat\rho} \arrow{u}[swap]{\tr_{\hat\rho}}
\end{tikzcd}\]
\caption{Overview of spaces and mappings.
\label{fig:overview}}
\end{figure}

Section \ref{sec:straight:paths} introduces basic tools,
and Section \ref{sec:geometry} explores various aspects of the geometry of digital images.
In Section \ref{sec:boundary:pairs}, we use the insights gathered to
prove that for every $\rho>0$, the mapping
\begin{equation}\label{def:tr}
\tr_\rho: S_\rho\to S_\rho^+\times S_\rho^+,\quad
\tr_\rho(M)=(\partial^0_\rho M,\partial^1_\rho M),
\end{equation}
is injective, and that 
\[\tr_\rho(S_\rho)=\bd_\rho,\] 
i.e.\ that the collection $\bd_\rho$ from Definition \ref{def:boundary:pairs}
is the space of all boundary pairs of sets from $S_\rho$.
In particular, its inverse $\tr_\rho^{-1}:\bd_\rho\to S_\rho$ is well-defined.
From then on, we consider fixed grids
\begin{equation}\label{fixed:grids}
\Delta_{\hat\rho}\subsetneq\Delta_{\check\rho},\quad
\check\rho,\hat\rho>0,\quad 
\hat\rho\in\check\rho\N_2.
\end{equation}
In Section \ref{sec:covering:theory}, we investigate the restriction and interpolation operators
\begin{align}
&R:S_{\check\rho}\to S_{\hat\rho},&&
R(M)=B_{\hat\rho/2}(M)\cap\Delta_{\hat\rho},&\label{def:R}\\
&I:S_{\hat\rho}\to S_{\check\rho},&&
I(M)=B_{\hat\rho/2}(M)\cap\Delta_{\check\rho}.&\label{def:I}
\end{align}
The operator $R$ is referred to as \emph{outer Jordan digitization}
\cite[Definition 2.8]{Klette} in digital geometry, and when $\hat\rho/\check\rho\in 2\N_1+1$,
the operator $I$ coincides with the \emph{subdivision operation} from computational dynamics
\cite[Section 3]{Dellnitz}.
We prefer them over other operator pairs for the following reasons:
\begin{itemize}
\item [i)] They are easily implementable (properties \eqref{R:union},
\eqref{R:Minkowski:sum}, \eqref{I:union} and \eqref{I:Minkowski:sum}).

\item [ii)] They have good approximation properties.
For any $\check M\in S_{\check\rho}$, the image $R(\check M)$
is the maximal best approximation to $\check M$ in $S_{\hat\rho}$ (properties
\eqref{best:approximation} and \eqref{maximal:best:approximation}), 
and $I$ satisfies a similar error estimate \eqref{I:Hausdorff}.
(For a brief discussion of why choosing the best approximation $I=\mathrm{id}$ 
is not an option, see Example \ref{ex:I}.)

\item [iii)] They map any $\check M\in S_{\check\rho}$ and $\hat M\in S_{\hat\rho}$
to the minimal elements of $V_{\check\rho}^{\hat\rho}(\check M)$ and 
$V_{\hat\rho}^{\check\rho}(\hat M)$, respectively (properties 
\eqref{voronoi:overapproximation}, \eqref{minimal:voronoi:overapproximation},
\eqref{I:voronoi:overapproximation} and \eqref{I:minimal:voronoi:overapproximation}).
This covering property is crucial for applications in computational dynamics 
and control theory such as \cite{Dellnitz} and \cite{Saint-Pierre}.

\item [iv)] They map any connected $\check M\in S_{\check\rho}$ and $\hat M\in S_{\hat\rho}$
to a connected set (properties \eqref{R:connected} and \eqref{I:connected}).
Algorithms such as the boundary tracking method from \cite{Rieger} cannot dispense
with connectedness.

\item [v)] The boundaries of their images are close to the boundaries of their preimages
(properties \eqref{new:boundary:in:old:boundary} and \eqref{I:new:boundary:in:old:boundary}).
It is essential for the aim of this paper, i.e.\ for working with boundaries instead of
full sets to reduce computational complexity, that the boundaries of the images are not 
more complicated than those of the preimages.

\item [vi)] We have $R\circ I=\mathrm{id}$ when ${\hat{\rho}}/\check{\rho}\in 2\N_1+1$ by
\eqref{I:R:id}, and we \emph{almost} have $R\circ I=\mathrm{id}$ when 
$\hat\rho/\check\rho\in 2\N_1$ by \eqref{I:R:almost:id}.
Theorem \ref{thm:no:better:covering:operators} shows that equality cannot be 
achieved in the latter case without sacrificing properties \eqref{voronoi:overapproximation},
\eqref{minimal:voronoi:overapproximation}, \eqref{I:voronoi:overapproximation} 
and \eqref{I:minimal:voronoi:overapproximation}.
\end{itemize}

In Section \ref{implementations}, we propose algorithms for the evaluation of the lifted
restriction and interpolation operators 
\begin{align}
&\partial R:\bd_{\check\rho}\to\bd_{\hat\rho},
&&\partial R:=\tr_{\hat\rho}\circ R\circ\tr_{\check\rho}^{-1},&\label{def:dR}\\
&\partial I:\bd_{\hat\rho}\to\bd_{\check\rho},
&&\partial I:=\tr_{\check\rho}\circ I\circ\tr_{\hat\rho}^{-1},&\label{def:dI}
\end{align}
which complete the diagram in Figure \ref{fig:overview}.
The algorithms do not evaluate the operators $\tr_{\check\rho}$, $\tr_{\hat\rho}$, $R$, $I$, 
$\tr_{\check\rho}^{-1}$ and $\tr_{\hat\rho}^{-1}$, but compute the desired boundary pair 
directly from the input.
A small computational example that illustrates the above commutative diagram is provided
in Figure \ref{fig:computational:example} at the end of the paper.

\section{Straight paths}\label{sec:straight:paths}

The proof of the following lemma is elementary.
Note that the usual rounding function does not have properties 
\eqref{round:plus:int} and \eqref{round:close:numbers}.

\begin{lemma}\label{lem:rounding}
The specific rounding function 
\[
\rd:\R\to\R,\quad
\rd(x):=\begin{cases}
\lfloor x\rfloor,&x-\lfloor x\rfloor\in[0,1/2),\\
\lceil x\rceil,&x-\lfloor x\rfloor\in[1/2,1),
\end{cases}
\]
has the following properties:
\begin{align}
&\forall\,k\in\Z:&&\rd(k)=k,&\label{rounding:integers:is:easy}\\
&\forall\,x\in\R,\ \forall\,k\in\Z:&&\rd(x)+k=\rd(x+k),&\label{round:plus:int}\\
&\forall\,x\in\R:&&|\rd(x)|\le\rd(|x|),&\label{round:and:abs}\\
&\forall\,x,z\in\R:&&(x\le z)\Rightarrow(\rd(x)\le\rd(z)).&\label{round:monotone}\\
&\forall\,x,z\in\R,\ \forall\,k\in\N:&&(|x-z|\le k)\Rightarrow|\rd(x)-\rd(z)|\le k.
\label{round:close:numbers}
\end{align}
\end{lemma}

The following path will be used frequently throughout the paper.
It is a special case of the digital ray from Definition 9.1 in \cite{Klette}.

\begin{lemma}\label{lem:grid:line}
Let $\rho>0$, let $x,z\in\Delta_\rho$, and let 
$k:=\|x-z\|_\infty/\rho$. 
Then the mapping $\phi(\,\cdot\,;x,z):[0,k]\to\Delta_\rho$ given by
$\phi(\,\cdot\,;x,z;\rho)=(x)$ when $x=z$ and 
\[\phi(\ell;x,z;\rho)_j:=\rd(\tfrac{k-\ell}{k}\tfrac{x_j}{\rho}
+\tfrac{\ell}{k}\tfrac{z_j}{\rho})\rho\quad\forall j\in[1,m],\ \forall\ell\in[0,k]\]
otherwise has the properties
\begin{align}
&\phi(0;x,z;\rho)=x\label{endpoints:0}\\
&\phi(k;x,z;\rho)=z\label{endpoints:1}\\
&\|\phi(\ell;x,z;\rho)-x\|_\infty=\ell\rho&&\forall\,\ell\in[0,k],&
\label{distances:to:endpoints:0}\\
&\|\phi(\ell;x,z;\rho)-z\|_\infty=(k-\ell)\rho&&\forall\,\ell\in[0,k],&
\label{distances:to:endpoints:1}\\
&\|\phi(\ell+1;x,z;\rho)-\phi(\ell;x,z;\rho)\|_\infty=\rho&&\forall\,\ell\in[0,k).&
\label{step:size:rho}
\end{align}
In particular, we have $\phi(\,\cdot\,;x,z;\rho)\in P_\rho(x,z)$ 
and $L(\phi(\,\cdot\,;x,z;\rho))=k$.
In addition, for every $y\in\Delta_\rho$, we have
\begin{equation}\label{path:stays:in:square}
\|\phi(\ell;x,z;\rho)-y\|_\infty\le\max\{\|x-y\|_\infty,\|y-z\|_\infty\}
\quad \forall\ell\in[0,k].
\end{equation}
\end{lemma}

\begin{proof}
Since $x_j/\rho\in\Z$ and $z_j/\rho\in\Z$ for all $j\in[1,m]$, 
statement \eqref{rounding:integers:is:easy} implies \eqref{endpoints:0} and \eqref{endpoints:1}.
\medskip

We show statement \eqref{distances:to:endpoints:0}.
Let $j\in[1,m]$.
Since $x_j/\rho\in\Z$, we obtain from \eqref{round:plus:int}, \eqref{round:and:abs} 
and \eqref{round:monotone} that
\begin{align*}
|\phi(\ell;x,z;\rho)_j-x_j|
&=|\rd(\tfrac{k-\ell}{k}\tfrac{x_j}{\rho}+\tfrac{\ell}{k}\tfrac{z_j}{\rho})
-\tfrac{x_j}{\rho}|\rho
=|\rd(\tfrac{\ell}{k}\tfrac{z_j-x_j}{\rho})|\rho\\
&\le\rd(\tfrac{\ell}{k}\tfrac{|z_j-x_j|}{\rho})\rho
\le\rd(\tfrac{\ell}{k}\tfrac{\|z-x\|_\infty}{\rho})\rho
=\rd(\ell)\rho=\ell\rho.
\end{align*}
By definition, there exists $j\in[1,m]$ with $z_j-x_j=\pm\|z-x\|_\infty$,
and the same computation yields for this $j$ that
\begin{align*}
|\phi(\ell;x,z;\rho)_j-x_j|
=|\rd(\tfrac{\ell}{k}\tfrac{\pm\|z-x\|_\infty}{\rho})|\rho
=|\rd(\pm\ell)|\rho
=|\pm\ell|\rho
=\ell\rho.
\end{align*}
All in all, we have shown that $\|\phi(\ell;x,z;\rho)-x\|_\infty=\ell\rho$.
A similar argument shows that statement \eqref{distances:to:endpoints:1}
holds as well.

\medskip

We show statement \eqref{step:size:rho}.
Let $j\in[1,m]$.
It follows from 
\eqref{round:plus:int} that
\begin{equation}\begin{aligned}
&\phi(\ell+1;x,z;\rho)_j-\phi(\ell;x,z;\rho)_j\\
&=\rd(\tfrac{x_j}{\rho}+\tfrac{\ell+1}{k}\tfrac{z_j-x_j}{\rho})\rho
-\rd(\tfrac{x_j}{\rho}+\tfrac{\ell}{k}\tfrac{z_j-x_j}{\rho})\rho\\
&=\rd(\tfrac{\ell+1}{k}\tfrac{z_j-x_j}{\rho})\rho
-\rd(\tfrac{\ell}{k}\tfrac{z_j-x_j}{\rho})\rho.
\end{aligned}\label{without:modulus}\end{equation}
Since
\[
|\tfrac{\ell+1}{k}\tfrac{z_j-x_j}{\rho}-\tfrac{\ell}{k}\tfrac{z_j-x_j}{\rho}|
=\tfrac{1}{k}\tfrac{|z_j-x_j|}{\rho}
\le\tfrac{1}{k}\tfrac{\|z-x\|_\infty}{\rho}
=1
\]
it follows from \eqref{without:modulus} and \eqref{round:close:numbers} that
\begin{align*}
&|\phi(\ell+1;x,z;\rho)_j-\phi(\ell;x,z;\rho)_j|
=|\rd(\tfrac{\ell+1}{k}\tfrac{z_j-x_j}{\rho})
-\rd(\tfrac{\ell}{k}\tfrac{z_j-x_j}{\rho})|\rho\le\rho.
\end{align*}
By definition, there exists $j\in[1,m]$ with $z_j-x_j=\pm\|z-x\|_\infty$,
for which
\begin{align*}
&\phi(\ell+1;x,z;\rho)_j-\phi(\ell;x,z;\rho)_j
=\rd(\tfrac{\ell+1}{k}\tfrac{z_j-x_j}{\rho})\rho
-\rd(\tfrac{\ell}{k}\tfrac{z_j-x_j}{\rho})\rho\\
&=\rd(\tfrac{\ell+1}{k}\tfrac{\pm\|z-x\|_\infty}{\rho})\rho
-\rd(\tfrac{\ell}{k}\tfrac{\pm\|z-x\|_\infty}{\rho})\rho
=\rd(\pm(\ell+1))\rho-\rd(\pm\ell)\rho
=\pm\rho,
\end{align*}
and thus statement \eqref{step:size:rho} holds.

\medskip

We check \eqref{path:stays:in:square}.
For every $j\in[1,m]$, statements \eqref{round:plus:int}, \eqref{round:and:abs} 
and \eqref{round:monotone} yield
\begin{align*}
|\phi(\ell;x,z;\rho)_j-y_j|
&=|\rd(\tfrac{k-\ell}{k}\tfrac{x_j}{\rho}+\tfrac{\ell}{k}\tfrac{z_j}{\rho})\rho-y_j|
=|\rd(\tfrac{k-\ell}{k}\tfrac{x_j}{\rho}+\tfrac{\ell}{k}\tfrac{z_j}{\rho})-\tfrac{y_j}{\rho}|\rho\\
&=|\rd(\tfrac{k-\ell}{k}\tfrac{x_j-y_j}{\rho}+\tfrac{\ell}{k}\tfrac{z_j-y_j}{\rho})|\rho
\le\rd(\tfrac{k-\ell}{k}\tfrac{|x_j-y_j|}{\rho}+\tfrac{\ell}{k}\tfrac{|z_j-y_j|}{\rho})\rho\\
&\le\rd(\tfrac{\max\{\|x-y\|_\infty,\|y-z\|_\infty\}}{\rho})\rho
=\max\{\|x-y\|_\infty,\|y-z\|_\infty\}.
\end{align*}

\end{proof}

We use the path $\phi$ to locate elements of $\partial^0_\rho M$ and $\partial^1_\rho M$.

\begin{lemma}\label{path:from:M:to:Mc:hits:boundaries}
Let $\rho>0$, let $M\in S_\rho^-$, let $x\in M$ and $z\in M^c\cap\Delta_\rho$,
and let $\phi(\,\cdot\,)=\phi(\,\cdot\,;x,z;\rho)$ be the path from Lemma \ref{lem:grid:line}.
Then the following statements hold:
\begin{itemize}
\item [a)] There exists $\ell\in[0,L(\phi))$ with
$\phi(\ell)\in\partial^0_\rho M$ and $\phi(\ell+1)\in\partial^1_\rho M$.
\item [b)] If $\|z-x\|_\infty=\dist(z,M)$, then we have $x\in\partial^0_\rho M$ and
$\phi(1)\in\partial^1_\rho M$ as well as $\|z-\phi(1)\|_\infty=\|z-x\|_\infty-\rho$
and $\|\phi(1)-x\|_\infty=\rho$.
\item [c)] If $\|x-z\|_\infty=\dist(x,M^c\cap\Delta_\rho)$, then 
$z\in\partial^1_\rho M$ and $\phi(L(\phi)-1)\in\partial^0_\rho M$ as well as 
$\|z-\phi(L(\phi)-1)\|_\infty=\rho$ and 
$\|\phi(L(\phi)-1)-x\|_\infty=\|x-z\|_\infty-\rho$.
\end{itemize}
\end{lemma}

\begin{proof}
a) Since $x\neq z$ we have $\|x-z\|_\infty\ge\rho$, and \eqref{step:size:rho}
implies that $L(\phi)>0$.
By \eqref{endpoints:0}, we have $\phi(0)=x\in M$, so there exists a maximal index 
$\ell\in[0,L(\phi)]$ with $\phi(\ell)\in M$.
By \eqref{endpoints:1}, we have 
\[\phi(L(\phi))=z\in M^c\cap\Delta_\rho,\]
and hence $\ell<L(\phi)$.
By maximality, we have $\phi(\ell+1)\in M^c\cap\Delta_\rho$.
We conclude from \eqref{step:size:rho}, \eqref{def:partial:1} and \eqref{def:partial:3} 
that $\phi(\ell)\in\partial^0_\rho M$ 
and $\phi(\ell+1)\in\partial^1_\rho M$.

\medskip

b) According to \eqref{distances:to:endpoints:1}, we have
\[\|z-\phi(1)\|_\infty=\|z-x\|_\infty-\rho 
=\dist(z,M)-\rho,\]
which implies $\phi(1)\in M^c\cap\Delta_\rho$.
By \eqref{endpoints:0} and \eqref{distances:to:endpoints:0}, we have
\[\|\phi(1)-x\|_\infty=\|\phi(1)-\phi(0)\|_\infty=\rho,\]
so by \eqref{def:partial:1} and \eqref{def:partial:3}, we have $x\in\partial^0_\rho M$ and
$\phi(1)\in\partial^1_\rho M$.

\medskip

The proof of part c) is similar to the proof of part b).
\end{proof}

\section{Geometry of digital images} \label{sec:geometry}

The following statement is trivial, but helps shorten some arguments.

\begin{lemma}\label{neighbour:in:boundary}
Let $\rho>0$, and let $M\in S_\rho$. 
Then
\begin{align*}
&\forall\,x\in M:
&&(x\in\partial^0_\rho M)\ \Leftrightarrow\ (\exists\,z\in\partial^1_\rho M\cap B_\rho(x)),&\\
&\forall\,z\in M^c:
&&(z\in\partial^1_\rho M)\ \Leftrightarrow\ (\exists\,x\in\partial^0_\rho M\cap B_\rho(z)).&
\end{align*}
\end{lemma}

We have the following alternative.

\begin{lemma}\label{alternative}
Let $\rho>0$ and let $M\in S_\rho$.
Then the following statements are equivalent:
(a) we have $M=\Delta_\rho$;
(b) we have $\partial^0_\rho M=\emptyset$ and $\partial^1_\rho M=\emptyset$;
(c) we have $\partial^0_\rho M=\emptyset$ or $\partial^1_\rho M=\emptyset$.
In particular, we have
\begin{align}
(M\in S_\rho^-)&\ \Leftrightarrow\ (\partial^0_\rho M\neq\emptyset\neq\partial^1_\rho M),
\label{S:minus:equiv:nonempty:partial}\\
(M=\Delta_\rho)&\ \Leftrightarrow\ (\partial^0_\rho M=\emptyset=\partial^1_\rho M).
\label{Delta:rho:equiv:empty:partial}
\end{align}
\end{lemma}

\begin{proof}
Statement a) implies that $M^c\cap\Delta_\rho=\emptyset$, which, in view of \eqref{def:partial:1}
and \eqref{def:partial:3}, implies b).
Statement b) clearly implies c).
Now assume that c) holds.
Since $M\in S_\rho$, there exists $x\in M$.
If there exists $z\in M^c\cap\Delta_\rho$, then Lemma \ref{path:from:M:to:Mc:hits:boundaries}a)
yields $\partial^0_\rho M\neq\emptyset\neq\partial^1_\rho M$.
This is false, so $M^c\cap\Delta_\rho=\emptyset$, which implies a).
The equivalences \eqref{S:minus:equiv:nonempty:partial} and \eqref{Delta:rho:equiv:empty:partial}
follow from a) through c).
\end{proof}

We investigate definitions \eqref{def:partial:1} and \eqref{def:partial:3} for $k=1$
under complementation.

\begin{lemma}\label{lem:simple:opposites}
Let $\rho>0$ and $M\in S_\rho^+$.
Then 
\begin{align}
&M=(M^c\cap\Delta_\rho)^c\cap\Delta_\rho,\label{double:complement}\\
&\partial^0_\rho M=\partial^1_\rho(M^c\cap\Delta_\rho),\label{d0d1}\\
&\partial^1_\rho M=\partial^0_\rho(M^c\cap\Delta_\rho)\label{d1d0}.
\end{align}
\end{lemma}

\begin{proof}
By the distributive law and by De Morgan's laws, we have
\[
M=M\cap\Delta_\rho
=(M\cap\Delta_\rho)\cup(\Delta_\rho^c\cap\Delta_\rho)
=(M\cup\Delta_\rho^c)\cap\Delta_\rho
=(M^c\cap\Delta_\rho)^c\cap\Delta_\rho,
\]
which is \eqref{double:complement}.
If $M=\emptyset$ or 
$M=\Delta_\rho$, then, by \eqref{Delta:rho:equiv:empty:partial}, \eqref{def:partial:1} 
and \eqref{def:partial:3}, we have 
\[
\partial^0_\rho M
=\partial^1_\rho M
=\partial^0_\rho(M^c\cap\Delta_\rho)
=\partial^1_\rho(M^c\cap\Delta_\rho)
=\emptyset,
\]
and statements \eqref{d0d1} and \eqref{d1d0} are trivial.
Now let $M\in S_\rho^-$.
We check \eqref{d0d1}.

\medskip

Let $x\in\partial^0_\rho M$.
Then by \eqref{double:complement}, 
we have $x\in M=(M^c\cap\Delta_\rho)^c\cap\Delta_\rho$, and by \eqref{def:partial:1} 
there exists $z\in M^c\cap\Delta_\rho$ with $\|x-z\|_\infty=\rho$.
Again by \eqref{def:partial:1}, we have $z\in\partial^0_\rho(M^c\cap\Delta_\rho)$, 
and it follows that $x\in\partial^1_\rho(M^c\cap\Delta_\rho)$.

\medskip

Conversely, let $x\in\partial^1_\rho(M^c\cap\Delta_\rho)$.
Then \eqref{def:partial:3} and \eqref{double:complement} imply $x\in M$, 
and by Lemma \ref{neighbour:in:boundary} there exists $z\in\partial^0_\rho(M^c\cap\Delta_\rho)$
with $\|x-z\|_\infty=\rho$.
Since $z\in M^c\cap\Delta_\rho$, it follows that $x\in\partial^0_\rho M$.

\medskip

All in all, statement \eqref{d0d1} holds.
Statement \eqref{d1d0} follows from \eqref{d0d1} with $M^c\cap\Delta_\rho$ in lieu of $M$,
and with \eqref{double:complement}.
\end{proof}

It is possible to distinguish between points in $M$ and $M^c\cap\Delta_\rho$
when only $\partial_\rho^0M$ and $\partial_\rho^1M$ are known.

\begin{theorem}\label{prop:xdists}
Let $\rho>0$, and let $M\in S_\rho^-$.
Then
\begin{align}
&\forall\,x\in M:
&&\dist(x,M^c\cap\Delta_\rho)=\dist(x,\partial_\rho^1M),&
\label{coarse:eq:xinMdist}\\
&\forall\,z\in M^c\cap\Delta_\rho:
&&\dist(z,M)=\dist(z,\partial_\rho^0M),&\label{coarse:eq:xnotinMdist}
\end{align}
and we have
\begin{align}
&\forall\,x\in\Delta_\rho:
&&(x\in M)\iff(\dist(x,\partial_\rho^0M)<\dist(x,\partial_\rho^1M)),&
\label{coarse:eq:xdists}\\
&\forall\,z\in\Delta_\rho:
&&(z\in M^c\cap\Delta_\rho)\iff(\dist(z,\partial_\rho^1M)<\dist(z,\partial_\rho^0M)).&
\label{opposite:coarse:eq:xdists}
\end{align}
\end{theorem}

\begin{proof}
First recall from \eqref{S:minus:equiv:nonempty:partial} that
$\partial^0_\rho M\neq\emptyset\neq\partial^1_\rho M$.

\medskip

We prove statement \eqref{coarse:eq:xnotinMdist}.
Let $z\in M^c\cap\Delta_\rho$ and take $x\in M$ with $\|z-x\|_\infty=\dist(z,M)$.
By Lemma \ref{path:from:M:to:Mc:hits:boundaries}b), 
we have $x\in\partial^0_\rho M$, and hence
\[
\dist(z,\partial^0_\rho M)
\le\|z-x\|_\infty
=\dist(z,M).
\]
Since $\partial^0_\rho M\subset M$ implies $\dist(z,M)\le\dist(z,\partial_\rho^0M)$,
statement \eqref{coarse:eq:xnotinMdist} holds.

\medskip

Since $M\in S_\rho^-$, we have $M^c\cap\Delta_\rho\in S_\rho^-$.
Hence statement \eqref{coarse:eq:xinMdist} follows from statement
\eqref{coarse:eq:xnotinMdist} with $M^c\cap\Delta_\rho$ in lieu of $M$
by using 
\eqref{double:complement} and \eqref{d1d0}.

\medskip

Now we prove that 
\begin{equation}\label{partial:result}
(z\in M^c\cap\Delta_\rho)\Rightarrow(\dist(z,\partial_\rho^1M)<\dist(z,\partial_\rho^0M)).
\end{equation}
Let $z\in M^c\cap\Delta_\rho$ and take $x\in\partial_\rho^0M$ with 
$\|z-x\|_\infty=\dist(x,\partial_\rho^0M)$.
By \eqref{coarse:eq:xnotinMdist}, we have
$\|z-x\|_\infty=\dist(z,M)$,
so by Lemma \ref{path:from:M:to:Mc:hits:boundaries}b), there exists
$y\in\partial^1_\rho M$ with $\|z-y\|_\infty=\|z-x\|_\infty-\rho$.
It follows, as desired, that
\[
\dist(z,\partial^1_\rho M)
\le\|z-y\|_\infty
<\|z-x\|_\infty
=\dist(z,\partial^0_\rho M).
\]

We may read statement \eqref{partial:result} with $M^c\cap\Delta_\rho$ in lieu of $M$.
Applying statements \eqref{double:complement}, \eqref{d0d1} and \eqref{d1d0} yields
\begin{equation}\label{other:partial:result}
(x\in M)\Rightarrow(\dist(x,\partial_\rho^0M)<\dist(x,\partial_\rho^1M)).
\end{equation}
Combining statements \eqref{partial:result} and \eqref{other:partial:result}
yields statements \eqref{coarse:eq:xdists} and \eqref{opposite:coarse:eq:xdists}.
\end{proof}

Statements \eqref{d0d1} and \eqref{d1d0} generalize to arbitrary indices.

\begin{proposition}
Let $\rho>0$ and $M\in S_\rho^-$.
Then 
\begin{equation}\label{flip:point:of:viev}
\partial_\rho^kM=\partial_\rho^{1-k}(M^c\cap\Delta_\rho)\quad\forall\,k\in\Z.
\end{equation}
\end{proposition}

\begin{proof}
When $k=0$, statement \eqref{flip:point:of:viev} is just statement \eqref{d0d1},
and when $k=1$, statement \eqref{flip:point:of:viev} is just \eqref{d1d0}.

\medskip

Let $k>1$. 
It follows from \eqref{def:partial:3} and \eqref{coarse:eq:xnotinMdist} that
\begin{equation}\label{other:representation:d:k}
\partial^k_\rho M=\{z\in M^c\cap\Delta_\rho:\dist(z,M)=k\rho\},
\end{equation}
and we know 
from \eqref{def:partial:4} and \eqref{d1d0} 
that
\begin{equation}\label{other:representation:d:1-k}
\partial^{1-k}_\rho(M^c\cap\Delta_\rho)
=\{z\in M^c\cap\Delta_\rho:\dist(z,\partial^1_\rho M)=(k-1)\rho\}.
\end{equation}

\medskip

We show $\partial_\rho^kM\subset\partial_\rho^{1-k}(M^c\cap\Delta_\rho)$.
Let $z\in\partial^k_\rho M$.
There exists $x\in\partial^0_\rho M$ with $\|z-x\|_\infty=\dist(z,\partial^0_\rho M)$.
Statements \eqref{coarse:eq:xnotinMdist} and \eqref{other:representation:d:k} yield
\begin{equation}\label{local:1}
\|z-x\|_\infty=\dist(z,\partial^0_\rho M)=\dist(z,M)=k\rho.
\end{equation}
By Lemma \ref{path:from:M:to:Mc:hits:boundaries}b), there exists
$y\in\partial^1_\rho M$ with $\|z-y\|_\infty=\|z-x\|_\infty-\rho$, so that
\[
\dist(z,\partial^1_\rho M)
\le\|z-y\|_\infty
=(k-1)\rho.
\]
Assume that there exists $z'\in\partial^1_\rho M$ with $\|z-z'\|_\infty<(k-1)\rho$.
Then, by Lemma \ref{neighbour:in:boundary}, there exists $x'\in\partial^0_\rho M$ with 
$\|z'-x'\|_\infty=\rho$.
It follows that 
\[\dist(z,\partial^0_\rho M)
\le\|z-z'\|_\infty+\|z'-x'\|_\infty<k\rho,\]
which contradicts \eqref{local:1}.
All in all, we have $\dist(z,\partial^1_\rho M)=(k-1)\rho$, and in view of 
\eqref{other:representation:d:1-k}, we have confirmed that 
$z\in\partial^{1-k}_\rho(M^c\cap\Delta_\rho)$.

\medskip

Now we show that $\partial_\rho^{1-k}(M^c\cap\Delta_\rho)\subset\partial_\rho^kM$.
Let $z\in\partial_\rho^{1-k}(M^c\cap\Delta_\rho)$.
By the triangle inequality, by \eqref{other:representation:d:1-k} and in view of \eqref{def:partial:3}, we have
\[\dist(z,\partial^0_\rho M)
\le\dist(z,\partial^1_\rho M)+\dist(\partial^1_\rho M,\partial^0_\rho M)
=(k-1)\rho+\rho
=k\rho.\]
Assume that $\dist(z,\partial^0_\rho M)<k\rho$.
Take $x\in\partial^0_\rho M$ with $\|z-x\|_\infty=\dist(z,\partial^0_\rho M)$.
By \eqref{coarse:eq:xnotinMdist} we have $\|z-x\|_\infty=\dist(z,M)$,
so by Lemma \ref{path:from:M:to:Mc:hits:boundaries}b), there exists $y\in\partial^1_\rho M$
with $\|z-y\|_\infty=\|z-x\|_\infty-\rho$, and hence
\[
\dist(z,\partial^1_\rho M)
\le\|z-y\|_\infty
=\|z-x\|_\infty-\rho
<(k-1)\rho,
\]
which contradicts \eqref{other:representation:d:1-k}.
Therefore $\dist(z,\partial^0_\rho M)=k\rho$, and hence $z\in\partial^k_\rho M$.

\medskip

Now let $k<0$.
But then $1-k\ge 0$, so we have already proved \eqref{flip:point:of:viev} with $1-k$ 
in lieu of $k$ and with $M^c\cap\Delta_\rho$ in lieu of $M$. 
Applying \eqref{double:complement} yields
\[\partial^{1-k}_\rho(M^c\cap\Delta_\rho)
=\partial^{1-(1-k)}_\rho((M^c\cap\Delta_\rho)^c\cap\Delta_\rho)
=\partial^k_\rho M.\]
\end{proof}

Both $\partial_\rho^kM$ and $\partial_\rho^{-k}M$ can be represented
in terms of the distance to $M$ and $M^c\cap\Delta_\rho$, respectively.

\begin{theorem}\label{thm:alternative:def:dkM}
Let $\rho>0$, and let $M\in S_\rho^-$.
Then 
\begin{align}
&\partial_\rho^kM=\{z\in M^c\cap\Delta_\rho:\dist(z,M)=k\rho\}&&\forall\,k\in\N_1,&
\label{dk:is:dist:k}\\
&\partial_\rho^{-k}M=\{x\in M:\dist(x,M^c\cap\Delta_\rho)=(k+1)\rho\}&&\forall\,k\in\N_0.&
\label{d:-k:is:dist:k+1}
\end{align}
\end{theorem}

\begin{proof}
Statement \eqref{dk:is:dist:k} follows directly from the definition \eqref{def:partial:3} 
and \eqref{coarse:eq:xnotinMdist}, 
and statement \eqref{d:-k:is:dist:k+1} follows from identities \eqref{flip:point:of:viev},
\eqref{dk:is:dist:k} and \eqref{double:complement}.
\end{proof}

Finally, it is possible to recover $\partial_{\rho}^0M$ and $\partial_{\rho}^1M$ from 
supersets.

\begin{lemma}\label{prop:bdrycleaning}
Let $\rho>0$, and let $M\in S_\rho$.
If two sets $H_0,H_1\subset\Delta_\rho$ satisfy
\begin{equation}\label{nested}
\partial_{\rho}^0M\subset H_0\subset M
\quad\text{and}\quad
\partial_\rho^1M\subset H_1\subset M^c\cap\Delta_\rho,
\end{equation}
then we have
\[
\partial_\rho^0M=\{x\in H_0:\dist(x,H_1)=\rho\}\ \;
\text{and}\ \
\partial_\rho^1M=\{z\in H_1:\dist(z,H_0)=\rho\}.
\]
\end{lemma}

\begin{proof}
If $M=\Delta_\rho$, then \eqref{Delta:rho:equiv:empty:partial} yields
$\partial_\rho^0M=\partial_\rho^1M=M^c\cap\Delta_\rho=\emptyset$, and \eqref{nested} 
implies that $H_1=\emptyset$.
It follows that
\[
\{x\in H_0:\dist(x,H_1)=\rho\}=\emptyset
\quad\text{and}\quad
\{z\in H_1:\dist(z,H_0)=\rho\}=\emptyset,
\]
so the desired statement holds.

Now let $M\in S_\rho^-$.
We prove that $\partial_\rho^0M=\{x\in H_0:\dist(x,H_1)=\rho\}$.
Let $x\in H_0$ with $\dist(x,H_1)=\rho$. 
Then $x\in M$ by \eqref{nested} and $x\in\partial^0_\rho M$ 
by \eqref{def:partial:1} because 
\[\dist(x,M^c\cap\Delta_\rho)\le\dist(x,H_1)=\rho.\]
Conversely, let $x\in\partial^0_\rho M$.
Then $x\in H_0$ by \eqref{nested}, and by Lemma \ref{neighbour:in:boundary} and \eqref{nested}
there exists $z\in\partial^1_\rho M\subset H_1$ with $\|x-z\|_\infty=\rho$.
Since $x\in M$, by \eqref{nested} and by the above, we have
\[\rho\le\dist(x,M^c\cap\Delta_\rho)\le\dist(x,H_1)\le\|x-z\|_\infty=\rho,\]
which implies that $\dist(x,H_1)=\rho$.

\medskip

The proof of the identity $\partial_\rho^1M=\{z\in H_1:\dist(z,H_0)=\rho\}$ is similar. 
\end{proof}

\section{Boundary pairs of digital images}\label{sec:boundary:pairs}

The following lemma refines Lemma \ref{alternative}.
Recall 
$\bd_\rho^-$ from Definition \ref{def:boundary:pairs}.

\begin{lemma}\label{Tr:from:S-:to:bd-}
Let $\rho>0$, and let $M\in S_\rho^-$.
Then $(\partial^0_\rho M,\partial^1_\rho M)\in\bd_\rho^-$.
\end{lemma}

\begin{proof}
By Lemma \ref{alternative}, the sets $\partial^0_\rho M$ and $\partial^1_\rho M$ 
satisfy axiom \eqref{bdry:axiom:0}.
Axiom \eqref{bdry:axiom:1} follows from \eqref{def:partial:1} 
and \eqref{def:partial:3}.
Lemma \ref{neighbour:in:boundary} yields axioms \eqref{bdry:axiom:2}
and \eqref{bdry:axiom:3}.

\medskip

We check axiom \eqref{bdry:axiom:4}.
Let $x\in\partial^0_\rho M$, let $z\in\partial^1_\rho M$, and let 
$p\in P_\rho(x,z)$ with $L(p)>1$.
Since $p(L(p))=z\in\partial^1_\rho M$, there exists a smallest index 
$\ell\in(0,L(p)]$ with $p(\ell)\in M^c\cap\Delta_\rho$.

Case 1: If $\ell=L(p)$, then $L(p)-1\in(0,L(p))$. 
By Definition \ref{def:concatenate}, we have $\|p(L(p)-1)-p(L(p))\|_\infty\le\rho$,
and since we have $p(L(p)-1)\in M$ and $p(L(p))\in M^c\cap\Delta_\rho$, it follows that 
$\|p(L(p)-1)-p(L(p))\|_\infty=\rho$.
Hence by \eqref{def:partial:1}, we have $p(L(p)-1)\in\partial^0_\rho M$.

Case 2: If $\ell<L(p)$, then $\ell\in(0,L(p))$. 
Again by Definition \ref{def:concatenate}, we have $\|p(\ell-1)-p(\ell)\|_\infty\le\rho$,
and since $p(\ell-1)\in M$ and $p(\ell)\in M^c\cap\Delta_\rho$, it follows that
$\|p(\ell-1)-p(\ell)\|_\infty=\rho$.
Hence by \eqref{def:partial:3}, we have $p(\ell)\in\partial^1_\rho M$.
\end{proof}

Conversely, we check that every pair $(D_0,D_1)\in\bd_\rho$ indeed corresponds to
the pair $(\partial^0_\rho M,\partial^0_\rho M)$ of a set $M\in S_\rho$.

\begin{definition}\label{def:M:D0:D1}
For any two sets $D_0,D_1\in\bd_\rho^- $ 
we define
\[M(D_0,D_1):=\{x\in\Delta_\rho:\dist(x,D_0)<\dist(x,D_1)\}.\]
\end{definition}

Given a pair $(D_0,D_1)\in\bd_\rho^-$, we can partition $\Delta_\rho$ into 
the disjoint union of the points closer to $D_0$ and the points closer to $D_1$.

\begin{lemma}\label{D0:D1:cover:Delta}
Let $(D_0,D_1)\in\bd_\rho^-$.
Then we have $M(D_0,D_1)\in S_\rho^-$ and $M(D_1,D_0)\in S_\rho^-$
as well as
\[D_0\subset M(D_0,D_1),\quad
D_1\subset M(D_1,D_0),\quad
M(D_1,D_0)=M(D_0,D_1)^c\cap\Delta_\rho.
\]
\end{lemma}

\begin{proof}
Definition \ref{def:M:D0:D1} clearly implies that
\begin{equation}\label{loc:10}
M(D_0,D_1)\cap M(D_1,D_0)=\emptyset.
\end{equation}
By \eqref{bdry:axiom:0} and \eqref{bdry:axiom:1}, 
and by Definition \ref{def:M:D0:D1}, we have 
$\emptyset\neq D_0\subset M(D_0,D_1)$ and $\emptyset\neq D_1\subset M(D_1,D_0)$.
Together with \eqref{loc:10}, this implies 
$M(D_0,D_1)\in S_\rho^-$ and $M(D_1,D_0)\in S_\rho^-$.

\medskip

Assume that there exists $y\in M(D_0,D_1)^c\cap M(D_1,D_0)^c\cap \Delta_\rho$. 
Then Definition \ref{def:M:D0:D1} yields $\dist(y,D_0)=\dist(y,D_1)$, and
there exist $k\in\N$, $x\in D_0$ and $z\in D_1$ with 
\begin{equation}\label{x:equal:distances}
\|x-y\|_\infty=\dist(y,D_0)=k\rho=\dist(y,D_1)=\|y-z\|_\infty.
\end{equation}
Since $D_0\cap D_1=\emptyset$ by axiom \eqref{bdry:axiom:0}, statement \eqref{x:equal:distances}
implies $k>0$.
With $\phi$ as in Lemma \ref{lem:grid:line}, consider the paths
$p_0:=\phi(\,\cdot\,;x,y;\rho)$ and $p_1:=\phi(\,\cdot\,;y,z;\rho)$
with $L(p_0)=L(p_1)=k$.
Since $p:=p_1\circ p_0\in P_\rho(x,z)$ and $L(p)=2k>1$, axiom \eqref{bdry:axiom:4} yields an index
$\ell\in(0,2k)$ with $p(\ell)\in D_0\cup D_1$.
By \eqref{concatenation}, and by statement \eqref{distances:to:endpoints:1} applied to $p_0$ 
and statement \eqref{distances:to:endpoints:0} applied to $p_1$, we have
\begin{align*}
\|p(\ell)-y\|_\infty
=\begin{cases}
\|p_0(\ell)-y\|_\infty,&\ell\in(0,k],\\\|y-p_1(\ell-k)\|_\infty,&\ell\in(k,2k)
\end{cases}
=\begin{cases}
(k-\ell)\rho,&\ell\in(0,k],\\(\ell-k)\rho,&\ell\in(k,2k).
\end{cases}
\end{align*}
Hence $\|p(\ell)-y\|_\infty<k\rho$, which
contradicts \eqref{x:equal:distances}.
Consequently, the initial assumption is false, and we have
\begin{equation}\label{loc:11}
M(D_0,D_1)^c\cap M(D_1,D_0)^c\cap \Delta_\rho=\emptyset.
\end{equation}
Combining \eqref{loc:10} and \eqref{loc:11} yields $M(D_1,D_0)=M(D_0,D_1)^c\cap\Delta_\rho$.
\end{proof}

The restriction of 
$\tr_\rho$ from \eqref{def:tr} to $S_\rho^-$ 
is a surjection onto $\bd_\rho^-$. 

\begin{lemma}\label{M:D0:D1:gives:bd}
For all $(D_0,D_1)\in\bd_\rho^-$,
we have $\partial^i_\rho M(D_0,D_1)=D_i$, $i=0,1$.
\end{lemma}

\begin{proof}
Let $M:=M(D_0,D_1)$.
By Lemma \ref{D0:D1:cover:Delta}, we have $M^c\cap\Delta_\rho=M(D_1,D_0)$
as well as $D_0\subset M$ and $D_1\subset M^c\cap\Delta_\rho$.
In view of statements \eqref{def:partial:1} and \eqref{def:partial:3},
properties \eqref{bdry:axiom:2} and \eqref{bdry:axiom:3} yield
$D_i\subset\partial^i_\rho M$, $i=0,1$.

\medskip

Conversely, let $x\in\partial^0_\rho M$.
By Lemma \ref{neighbour:in:boundary}, there exists $z\in\partial^1_\rho M$ 
with $\|x-z\|_\infty=\rho$.
Since
\[x\in\partial^0_\rho M\subset M=M(D_0,D_1)
\quad\text{and}\quad
z\in\partial^1_\rho M\subset M^c\cap\Delta_\rho=M(D_1,D_0),\]
we obtain using Definition \ref{def:M:D0:D1} and the triangle inequality that
\begin{align}
&\dist(x,D_0)<\dist(x,D_1)\le\|x-z\|_\infty+\dist(z,D_1)=\rho+\dist(z,D_1),
\label{dist:x:D0}\\
&\dist(z,D_1)<\dist(z,D_0)\le\|z-x\|_\infty+\dist(x,D_0)=\rho+\dist(x,D_0).
\label{dist:z:D1}
\end{align}
In the following, note that $\dist(x,D_0)\in\rho\N_0$ and $\dist(z,D_1)\in\rho\N_0$.
Combining \eqref{dist:x:D0} and \eqref{dist:z:D1}, we see that 
$\dist(x,D_0)\le\dist(z,D_1)\le\dist(x,D_0)$,
and hence there exists $k\in\N$ with
\begin{equation}\label{D:free:zone:1}
\dist(x,D_0)=\dist(z,D_1)=k\rho.
\end{equation}
Substituting \eqref{D:free:zone:1} back into \eqref{dist:x:D0} and \eqref{dist:z:D1} yields
\begin{align*}
&k\rho<\dist(x,D_1)\le(k+1)\rho,\\
&k\rho<\dist(z,D_0)\le(k+1)\rho,
\end{align*}
which implies
\begin{equation}\label{D:free:zone:2}
\dist(x,D_1)=\dist(z,D_0)=(k+1)\rho.
\end{equation}
Assume that $k>0$.
By \eqref{D:free:zone:1}, there exist $x'\in D_0$ and $z'\in D_1$ with
\[\|x-x'\|_\infty=k\rho=\|z-z'\|_\infty.\]
Let $\phi$ as in Lemma \ref{lem:grid:line}, and consider the  paths
$p_0:=\phi(\,\cdot\,;x',x;\rho)$, $p_1:=(x,z)$ and $p_2:=\phi(\,\cdot\,;z,z';\rho)$
with $L(p_0)=L(p_2)=k$ and $L(p_1)=1$.
Since 
\[
p:=p_2\circ p_1\circ p_0\in P_\rho(x',z')
\quad\text{and}\quad
L(p)=2k+1>1,
\] 
axiom \eqref{bdry:axiom:4} yields an index $\ell\in(0,2k+1)$ with $p(\ell)\in D_0\cup D_1$.
If $\ell\in(0,k]$, then statement \eqref{distances:to:endpoints:1} yields
\[
\|p(\ell)-x\|_\infty
=\|p_0(\ell)-x\|_\infty
\le(k-\ell)\rho
\le k\rho.
\]
Alternatively, if $\ell\in[k+1,2k+1)$, then statement \eqref{distances:to:endpoints:0} yields
\begin{align*}
\|p(\ell)-x\|_\infty
&=\|x-p_2(\ell-k-1)\|_\infty\\
&\le\|x-z\|_\infty+\|z-p_2(\ell-k-1)\|_\infty
\le\rho+(\ell-k-1)\rho
\le k\rho.
\end{align*}
We summarize that
\begin{equation}\label{point:of:view:x}
\|p(\ell)-x\|_\infty\le k\rho,
\end{equation}
and the same computation with reversed roles yields
\begin{equation}\label{point:of:view:z}
\|p(\ell)-z\|_\infty\le k\rho.
\end{equation}
Finally, we observe that when $p(\ell)\in D_0$, then \eqref{point:of:view:z} contradicts
\eqref{D:free:zone:2}, and when $p(\ell)\in D_1$, then \eqref{point:of:view:x} contradicts
\eqref{D:free:zone:2}.
All in all, the assumption $k>0$ is false.
Hence $k=0$, and by \eqref{D:free:zone:1}, we have $x\in D_0$.
This shows $\partial^0_\rho M\subset D_0$.

\medskip

The same argument shows that $\partial^1_\rho M\subset D_1$ holds as well.
\end{proof}

The following theorem is the main result of this section.
Recall the operator $\tr_\rho$ from \eqref{def:tr}.

\begin{theorem}\label{thm:F:well:defined}
We have $\tr_\rho(S_\rho)=\bd_\rho$,
and 
$\tr_\rho:S_\rho\to\bd_\rho$ is a bijection with inverse
\[\tr_\rho^{-1}(D_0,D_1)
=\begin{cases}M(D_0,D_1),&(D_0,D_1)\in\bd_\rho^-,\\
\Delta_\rho,&(D_0,D_1)=(\emptyset,\emptyset).
\end{cases}\]
\end{theorem}

\begin{proof}
By Lemma \ref{Tr:from:S-:to:bd-}, we have $\tr_\rho(S_\rho^-)\subset\bd_\rho^-$,
and by Lemma \ref{M:D0:D1:gives:bd}, we have $\bd_\rho^-\subset\tr_\rho(S_\rho^-)$.
Hence the operator $\tr_\rho:S_\rho^-\to\bd_\rho^-$ is surjective,
and by statement \eqref{coarse:eq:xdists}, it is injective.
In view of Lemma \ref{alternative}, 
we have $\tr_\rho(S_\rho)=\bd_\rho$, 
and that $\tr_\rho:S_\rho\to\bd_\rho$ is a bijection.
Finally, by Lemmas \ref{M:D0:D1:gives:bd} and \ref{alternative}, the inverse
admits the above representation.
\end{proof}

\section{Restriction and interpolation operators}
\label{sec:covering:theory}

From now on, we work with the two distinct nested grids from \eqref{fixed:grids}.

\medskip

We present an auxiliary result that will be used in Theorems \ref{thm:R:properties} 
and \ref{thm:refining:works}.

\begin{lemma}\label{lem:v:cover}
For all $k\in\N_0$, all $\check x\in\Delta_{\check\rho}$ and all $v\in B_{(k+1)\check\rho/2}(\check x)$,
there exists 
\[\hat x\in B_{(\hat\rho+k\check\rho)/2}(\check x)\cap B_{\hat\rho/2}(v)\cap\Delta_{\hat\rho}.\]
\end{lemma}

\begin{proof}
Consider the point $\hat x\in\Delta_{\hat\rho}$ given by
\[
\hat x_j:=\begin{cases}
\lfloor\tfrac{v_j}{\hat\rho}\rfloor\hat\rho,
&\tfrac{v_j}{\hat\rho}-\lfloor\tfrac{v_j}{\hat\rho}\rfloor<\tfrac12,\\
\lfloor\tfrac{v_j}{\hat\rho}\rfloor\hat\rho,
&\tfrac{v_j}{\hat\rho}-\lfloor\tfrac{v_j}{\hat\rho}\rfloor=\tfrac12,\ \check x_j<v_j,\\
\lceil\tfrac{v_j}{\hat\rho}\rceil\hat\rho,
&\tfrac{v_j}{\hat\rho}-\lfloor\tfrac{v_j}{\hat\rho}\rfloor=\tfrac12,\ v_j\le\check x_j,\\
\lceil\tfrac{v_j}{\hat\rho}\rceil\hat\rho,
&\tfrac{v_j}{\hat\rho}-\lfloor\tfrac{v_j}{\hat\rho}\rfloor>\tfrac12
\end{cases}
\]
for all $j\in[1,m]$. 
We first check that for all $j\in[1,m]$, we have
\begin{align} 
&|\hat x_j-\check x_j|<\tfrac{\hat\rho}{2}+\tfrac{(k+1)\check\rho}{2},\label{aim:1}\\
&|v_j-\hat x_j|\le\tfrac{\hat\rho}{2}.\label{aim:2}
\end{align}
Case 1: When $\tfrac{v_j}{\hat\rho}-\lfloor\tfrac{v_j}{\hat\rho}\rfloor<\tfrac12$, 
then \eqref{aim:1} follows from
$\hat x_j-\check x_j=\lfloor\tfrac{v_j}{\hat\rho}\rfloor\hat\rho-\check x_j$
and
\[
-\tfrac{(k+1)\check\rho}{2}-\tfrac{\hat\rho}{2}
\le v_j-\tfrac{\hat\rho}{2}-\check x_j
<\lfloor\tfrac{v_j}{\hat\rho}\rfloor\hat\rho-\check x_j
\le v_j-\check x_j
\le\tfrac{(k+1)\check\rho}{2}.\]
Inequality \eqref{aim:2} holds because
$|v_j-\hat x_j|
=|v_j-\lfloor\tfrac{v_j}{\hat\rho}\rfloor\hat\rho|
=|\tfrac{v_j}{\hat\rho}-\lfloor\tfrac{v_j}{\hat\rho}\rfloor|\hat\rho
<\tfrac{\hat\rho}{2}$.

Case 2: When $\tfrac{v_j}{\hat\rho}-\lfloor\tfrac{v_j}{\hat\rho}\rfloor=\tfrac12$ 
and $\check x_j<v_j$, then inequality \eqref{aim:1} follows from
$\hat x_j-\check x_j
=\lfloor\tfrac{v_j}{\hat\rho}\rfloor\hat\rho-\check x_j
=v_j-\check x_j-\tfrac{\hat\rho}{2}$
and
\[
-\tfrac{\hat\rho}{2}
<v_j-\check x_j-\tfrac{\hat\rho}{2}
<v_j-\check x_j\le\tfrac{(k+1)\check\rho}{2}.
\]
Inequality \eqref{aim:2} holds because
$|v_j-\hat x_j|
=|\tfrac{v_j}{\hat\rho}-\lfloor\tfrac{v_j}{\hat\rho}\rfloor|\hat\rho
=\tfrac{\hat\rho}{2}$.

Case 3: When $\tfrac{v_j}{\hat\rho}-\lfloor\tfrac{v_j}{\hat\rho}\rfloor=\tfrac12$
and $v_j\le\check x_j$, then inequality \eqref{aim:1} follows from
$\hat x_j-\check x_j
=\lceil\tfrac{v_j}{\hat\rho}\rceil\hat\rho-\check x_j
=v_j+\tfrac{\hat\rho}{2}-\check x_j$
and
\[
-\tfrac{(k+1)\check\rho}{2}+\tfrac{\hat\rho}{2}
\le v_j+\tfrac{\hat\rho}{2}-\check x_j
\le\tfrac{\hat\rho}{2}.
\]
Inequality \eqref{aim:2} follows from
$|v_j-\hat x_j|
=|\tfrac{v_j}{\hat\rho}-\lceil\tfrac{v_j}{\hat\rho}\rceil|\hat\rho
=\tfrac{\hat\rho}{2}$.

Case 4: When $\tfrac{v_j}{\hat\rho}-\lfloor\tfrac{v_j}{\hat\rho}\rfloor>\tfrac12$,
then \eqref{aim:1} follows from
$\hat x_j-\check x_j
=\lceil\tfrac{v_j}{\hat\rho}\rceil\hat\rho-\check x_j$
and
\[
-\tfrac{(k+1)\check\rho}{2}
\le v_j-\check x_j
<\lceil\tfrac{v_j}{\hat\rho}\rceil\hat\rho-\check x_j
<v_j+\tfrac{\hat\rho}{2}-\check x_j
\le\tfrac{(k+1)\check\rho}{2}+\tfrac{\hat\rho}{2}.
\]
Inequality \eqref{aim:2} follows from
$|v_j-\hat x_j|
=|\tfrac{v_j}{\hat\rho}-\lceil\tfrac{v_j}{\hat\rho}\rceil|\hat\rho
<\tfrac{\hat\rho}{2}$.

Since $\|\hat x-\check x\|_\infty\in\check\rho\N_0$, the statement of the 
lemma follows from inequalities \eqref{aim:1} and \eqref{aim:2}.
\end{proof}

The operator $R$ has several properties that make it stand out
among the operators that map $S_{\check\rho}$ to $S_{\hat\rho}$.
Some of the notation is from Definition \ref{def:collections}.

\begin{theorem}\label{thm:R:properties}
The operator $R:S_{\check\rho}\to S_{\hat\rho}$ satisfies 
$R(\Delta_{\check\rho})=\Delta_{\hat\rho}$.
It has the set-theoretical properties
\begin{align}
&\forall\,\check M,\check M'\in S_{\check\rho}:
&&R(\check M\cup\check M')=R(\check M)\cup R(\check M'),&
\label{R:union}\\
&\forall\,\check M,\check M'\in S_{\check\rho}:
&&(\check M\subset\check M')\Rightarrow(R(\check M)\subset R(\check M')),&
\label{R:monotone}\\
&\forall\,\check M\in S_{\check\rho},\ \forall\,\hat x\in\Delta_{\hat\rho}:
&&R(\hat x+\check M)=\hat x+R(\check M),&
\label{R:Minkowski:sum}
\end{align}
the approximation properties
\begin{align}
&\forall\,\check M\in S_{\check\rho}:
&&\dist_H(R(\check M),\check M)\le\hat\rho/2,&
\label{eq:GammaHausdorff}\\
&\forall\,\check M\in S_{\check\rho}:
&&R(\check M)\in A_{\hat\rho}(\check M),&
\label{best:approximation}\\
&\forall\,\check M\in S_{\check\rho},\ 
\forall\,\hat M\in A_{\hat\rho}(\check M):&&\hat M\subset R(\check M),&
\label{maximal:best:approximation}\\
&\forall\,\check M\in S_{\check\rho}:
&&R(\check M)\in V_{\check\rho}^{\hat\rho}(\check M),&
\label{voronoi:overapproximation}\\
&\forall\,\check M\in S_{\check\rho},\ 
\forall\,\hat M\in V_{\check\rho}^{\hat\rho}(\check M):
&&R(\check M)\subset\hat M,&
\label{minimal:voronoi:overapproximation}
\end{align}
and the topology-related properties
\begin{align}
&\forall\,\check M\in C_{\check\rho}:
&&R(\check M)\in C_{\hat\rho},&\label{R:connected}\\
&\forall\,\check M\in S_{\check\rho}:
&&\partial^0_{\hat\rho}R(\check M)\subset R(\partial^0_{\check\rho}\check M).&
\label{new:boundary:in:old:boundary}
\end{align}
\end{theorem}

\begin{proof}
The proofs of statements \eqref{R:union}, \eqref{R:monotone} and \eqref{R:Minkowski:sum} 
are elementary, and it is obvious that $R(\Delta_{\check\rho})=\Delta_{\hat\rho}$.

\medskip

We check \eqref{eq:GammaHausdorff}.
Let $\check M\in S_{\check\rho}$.
By \eqref{def:R}, we have 
\[\dist(R(\check M),\check M)\le\hat\rho/2.\]
Conversely, let $\check x\in\check M$.
Then the point $\hat x\in\Delta_{\hat\rho}$ given by $\hat x_j=\rd(\check x_j/\hat\rho)\hat\rho$
for all $j\in[1,m]$ satisfies $\|\hat x-\check x\|_\infty\le\hat\rho/2$, which implies 
$\hat x\in R(\check M)$.
Hence
\[\dist(\check M,R(\check M))\le\hat\rho/2,\] 
and the proof of \eqref{eq:GammaHausdorff} is complete. 
The fact that $\hat x\in R(M)$ also shows that $R(\check M)\neq\emptyset$, 
and hence that $R$ is well-defined.

\medskip

In the following, we check \eqref{best:approximation} and \eqref{maximal:best:approximation}.
First note that for every $\check M\in S_{\check\rho}$, we have
$\emptyset\neq\{\dist_H(\hat M,\check M):\hat M\in S_{\hat\rho}\}\subset\check\rho\N_0$,
and it follows from the well-ordering principle that
\begin{equation}\label{A:nonempty}
\forall\,\check M\in S_{\check\rho}:\quad A_{\hat\rho}(\check M)\neq \emptyset.
\end{equation}

Assume that \eqref{maximal:best:approximation} is false.
Then there exist $\check M\in S_{\check\rho}$, 
$\hat M\in A_{\hat\rho}(\check M)$ and $\hat x\in\hat M\cap(R(\check M))^c$.
Hence $\hat x\notin B_{\hat\rho/2}(\check M)$, and we find
\[\dist(\hat M,\check M)\ge\dist(\hat x,\check M)>\hat\rho/2
\ge\dist_H(R(\check M),\check M),\] 
which contradicts $\hat M\in A_{\hat\rho}(\check M)$.
All in all, statement \eqref{maximal:best:approximation} holds.

Assume that \eqref{best:approximation} is false.
Then there is $\check M\in S_{\check\rho}$ with
$R(\check M)\notin A_{\hat\rho}(\check M)$.
By 
\eqref{A:nonempty}, there exists $\hat M\in A_{\hat\rho}(\check M)$,
and by \eqref{def:A}, since $R(\check M)\notin A_{\hat\rho}(\check M)$,
and by \eqref{eq:GammaHausdorff}, we have 
\[\dist_H(\check M,\hat M)<\dist_H(\check M,R(\check M))\le\hat\rho/2.\]
The triangle inequality, statement \eqref{eq:GammaHausdorff}, and the above
inequality yield
\[\dist(R(\check M),\hat M)\le\dist(R(\check M),\check M)+\dist(\check M,\hat M)<\hat\rho.\]
Since $\dist(R(\check M),\hat M)\in\hat\rho\N_0$,
this implies $\dist(R(\check M),\hat M)=0$ and $R(\check M)\subset\hat M$.
By \eqref{maximal:best:approximation} we also have $\hat M\subset R(\check M)$,
and hence $\hat M=R(\check M)$, which contradicts the definition of these sets.
All in all, statement \eqref{best:approximation} holds.

\medskip

In view of \eqref{def:R} and \eqref{def:V}, statement \eqref{voronoi:overapproximation} 
is equivalent with
\[B_{\check\rho/2}(\check M)\subset B_{\hat\rho/2}(B_{\hat\rho/2}(\check M)\cap\Delta_{\hat\rho}))
\quad\forall\check M\in S_{\check\rho}.\]
To check the above inclusion,
let $\check M\in S_{\check\rho}$ and $v\in B_{\check\rho/2}(\check M)$.
Then there exists $\check x\in\check M$ with $v\in B_{\check\rho/2}(\check x)$.
By Lemma \ref{lem:v:cover} (with $k=0$), there exists $\hat x\in B_{\hat\rho/2}(\check x)\cap\Delta_{\hat\rho}$
with $v\in B_{\hat\rho/2}(\hat x)$.
But then, as desired, we have
\[v\in B_{\hat\rho/2}(\hat x)
\subset B_{\hat\rho/2}(B_{\hat\rho/2}(\check x)\cap\Delta_{\hat\rho}) 
\subset B_{\hat\rho/2}(B_{\hat\rho/2}(\check M)\cap\Delta_{\hat\rho})).\] 

\medskip

We show \eqref{minimal:voronoi:overapproximation}.
Let $\check M\in S_{\check\rho}$, let $\hat M\in S_{\hat\rho}$,
and assume that there exists 
$\hat x\in R(\check M)\cap\hat M^c$.
By \eqref{def:R}, there exists $\check x\in\check M$ with 
$\|\check x-\hat x\|_\infty\le\hat\rho/2$, 
and the point
\[v:=\begin{cases}\check x+\tfrac{\check\rho}{2}
\tfrac{\hat x-\check x}{\|\check x-\hat x\|_\infty},&\check x\neq\hat x,\\
\hat x,&\check x=\hat x\end{cases}\]
satisfies $v\in B_{\check\rho/2}(\check x)$ and hence
\begin{equation}\label{in:Voronoi:M}
v\in B_{\check\rho/2}(\check M).
\end{equation}
If $\check x=\hat x$, then $\|v-\hat x\|_\infty=0$.
If $\check x\neq\hat x$, then $\|\check x-\hat x\|_\infty\ge\check\rho$ and
\begin{align*}
\|v-\hat x\|_\infty
=\|(1-\tfrac{\check\rho}{2\|\check x-\hat x\|_\infty})(\check x-\hat x)\|_\infty
<\|\check x-\hat x\|_\infty\le\hat\rho/2.
\end{align*}
In both cases, we have $\|v-\hat x\|_\infty<\hat\rho/2$. 
Because of $\hat x\in\hat M^c\cap\Delta_{\hat\rho}$, 
we have $\dist(\hat x,\hat M)\ge\hat\rho$, 
and the triangle inequality yields
\[\dist(v,\hat M)\ge\dist(\hat x,\hat M)-\|v-\hat x\|_\infty>\hat\rho/2.\]
Hence we have
\begin{equation}\label{not:in:Voronoi:M'}
v\in B_{\hat\rho/2}(\hat M)^c.
\end{equation}
In view of \eqref{def:V}, statements \eqref{in:Voronoi:M} and \eqref{not:in:Voronoi:M'}
mean that $\hat M\notin V_{\check\rho}^{\hat\rho}(\check M)$.
This completes the proof of \eqref{minimal:voronoi:overapproximation}.

\medskip

We check \eqref{R:connected}.
Let $\check M\in C_{\check\rho}$, and let $\hat x\in R(\check M)$ and $\hat x'\in R(\check M)$.
By \eqref{def:R}, there exist $\check x\in\check M$ and $\check x'\in\check M$ 
with $\|\hat x-\check x\|_\infty\le\hat\rho/2$
and $\|\hat x'-\check x'\|_\infty\le\hat\rho/2$.
Since $\check M\in C_{\check\rho}$, there exists $\check p\in P_{\check\rho}(\check x,\check x')$ 
with $\check p(\ell)\in\check M$ for all $\ell\in[0,L(\check p)]$.
Let $\hat\xi=(\hat\xi_0,\ldots,\hat\xi_{L(\check p)})\in(\Delta_{\hat\rho})^{L(\check p)+1}$ 
be given by
$\hat\xi_0:=\hat x$, $\hat\xi_{L(\check p)}:=\hat x'$, and 
\[(\hat\xi_\ell)_j:=\rd(\check p(\ell)_j/\hat\rho)\hat\rho
\quad\forall\,\ell\in(0,L(\check p)),\ \forall\,j\in[1,m]\]
with $\rd$ as in Lemma \ref{lem:rounding}.
Then $\|\hat \xi_\ell-\check p(\ell)\|_\infty\le\hat\rho/2$ for all $\ell\in[0,L(\check p)]$,
and since $\check p(\ell)\in\check M$ for all $\ell\in[0,L(p)]$, we have 
$\hat\xi_\ell\in R(\check M)$ for all $\ell\in[0,L(p)]$.
In addition, we have
\[
\|\hat\xi_\ell-\hat\xi_{\ell-1}\|_\infty
\le\|\hat\xi_\ell-\check p(\ell)\|_\infty
+\|\check p(\ell)-\check p(\ell-1)\|_\infty
+\|\check p(\ell-1)-\hat\xi_{\ell-1}\|_\infty
\le\hat\rho+\check\rho
\]
for all $\ell\in(0,L(\check p)]$.
Since $\hat\xi_\ell\in\Delta_{\hat\rho}$ for all $\ell\in[0,L(\check p)]$, it follows that
\[\|\hat\xi_\ell-\hat\xi_{\ell-1}\|_\infty\le\hat\rho\quad\forall\,\ell\in(0,L(\check p)].\]
Hence 
$\hat p:=(\hat\xi_0,\hat\xi_1,\ldots,\hat\xi_{L(\check p)})\in(\Delta_{\hat\rho})^{L(\check p)+1}$
is a path $\hat p\in P_{\hat\rho}(\hat x,\hat x')$ of length $L(\hat p)=L(\check p)$ and 
$\hat p(\ell)\in R(\check M)$
for all $\ell\in[0,L(\hat p)]$.
Since $\hat x$ and $\hat x'$ were arbitrary, 
it follows that $R(\check M)\in C_{\hat\rho}$.
All in all, statement \eqref{R:connected} holds.

\medskip

Finally, we show \eqref{new:boundary:in:old:boundary}.
Let $\check M\in S_{\check\rho}$ and let $\hat x\in\partial^0_{\hat\rho} R(\check M)$.
By \eqref{def:partial:1}, there exists $\hat z\in R(\check M)^c\cap\Delta_{\hat\rho}$ 
with $\|\hat z-\hat x\|_\infty=\hat\rho$. 
By \eqref{def:R}, 
there exists $\check x\in\check M$ with $\|\check x-\hat x\|_\infty\le\hat\rho/2$.

Case 1: If $\hat\rho/\check\rho\in 2\N_1$, then let
$\check z:=(\hat x+\hat z)/2$.
Since both $\check z\in\Delta_{\check\rho}$ and $\|\check z-\hat z\|_\infty=\hat\rho/2$, 
it follows from 
$\hat z\in R(\check M)^c\cap\Delta_{\hat\rho}$
and \eqref{def:R} that $\check z\in\check M^c\cap\Delta_{\check\rho}$.
Let $\phi(\,\cdot\,)=\phi(\,\cdot\,;\check x,\check z;\check\rho)$ 
be the path from Lemma \ref{lem:grid:line}.
By Lemma \ref{path:from:M:to:Mc:hits:boundaries}a), there exists $\ell\in[0,L(\phi))$ with
$\phi(\ell)\in\partial^0_{\check\rho}\check M$.
Because of $\|\check x-\hat x\|_\infty\le\hat\rho/2$ and 
$\|\hat x-\check z\|_\infty=\hat\rho/2$, and by \eqref{path:stays:in:square}, we have
\[\|\phi(\ell)-\hat x\|_\infty
\le\max\{\|\check x-\hat x\|_\infty,\|\hat x-\check z\|_\infty\}
=\hat\rho/2.\]
Again with \eqref{def:R}, it follows that 
$\hat x\in R(\partial^0_{\check\rho}\check M)$.

Case 2: If $\hat\rho/\check\rho\in 2\N_1+1$, then let $n\in\N_1$ with
$\hat\rho=(2n+1)\check\rho$.
Consider the point
$\check x':=\tfrac{\hat\rho+\check\rho}{2\hat\rho}\hat x
+\tfrac{\hat\rho-\check\rho}{2\hat\rho}\hat z$.
Since $\hat z_j-\hat x_j\in\hat\rho\Z$ for all $j\in[1,m]$, we have
\[
\check x_j'
=\tfrac{\hat\rho+\check\rho}{2\hat\rho}\hat x_j
+\tfrac{\hat\rho-\check\rho}{2\hat\rho}\hat z_j
=\hat x_j+\tfrac{\hat\rho-\check\rho}{2\hat\rho}(\hat z_j-\hat x_j)
=\hat x_j+\tfrac{n\check\rho}{\hat\rho}(\hat z_j-\hat x_j)\in\check\rho\Z
\]
for all $j\in[1,m]$, and hence $\check x'\in\Delta_{\check\rho}$.
We also compute
\[
|\hat x_j-\check x'_j|
=|\hat x_j-\tfrac{\hat\rho+\check\rho}{2\hat\rho}\hat x_j
-\tfrac{\hat\rho-\check\rho}{2\hat\rho}\hat z_j|
=\tfrac{\hat\rho-\check\rho}{2\hat\rho}|\hat x_j-\hat z_j|
\quad\forall\,j\in[1,m],
\]
which implies $\|\hat x-\check x'\|_\infty
=\tfrac{\hat\rho-\check\rho}{2\hat\rho}\|\hat x-\hat z\|_\infty
=(\hat\rho-\check\rho)/2$.

Case 2a: If $\check x'\in\check M$, consider 
$\check z':=\tfrac{\hat\rho-\check\rho}{2\hat\rho}\hat x
+\tfrac{\hat\rho+\check\rho}{2\hat\rho}\hat z$.
Similar computations as above show that $\check z'\in\Delta_{\check\rho}$, 
that $\|\check z'-\hat z\|_\infty=(\hat\rho-\check\rho)/2$, 
and that $\|\check z'-\check x'\|_\infty=\check\rho$. 
Since $\|\check z'-\hat z\|_\infty<\hat\rho/2$, it follows from 
$\hat z\in R(\check M)^c\cap\Delta_{\hat\rho}$
and \eqref{def:R} that $\check z'\in\check M^c\cap\Delta_{\check\rho}$.
Hence $\|\check z'-\check x'\|_\infty=\check\rho$ and 
\eqref{def:partial:1} yield 
$\check x'\in\partial^0_{\check\rho}\check M$.
Since $\|\check x'-\hat x\|_\infty<\hat\rho/2$, it follows with \eqref{def:R} 
that $\hat x\in R(\partial^0_{\check\rho}\check M)$.

Case 2b: If $\check x'\in\check M^c\cap\Delta_{\check\rho}$, then 
let $\phi(\,\cdot\,)=\phi(\,\cdot\,;\check x,\check x';\check\rho)$ be the path from Lemma
\ref{lem:grid:line}.
By Lemma \ref{path:from:M:to:Mc:hits:boundaries}a), there exists $\ell\in[0,L(\phi))$ with
$\phi(\ell)\in\partial^0_{\check\rho}\check M$, and by \eqref{path:stays:in:square}, we have
\[\|\phi(\ell)-\hat x\|_\infty
\le\max\{\|\check x-\hat x\|_\infty,\|\hat x-\check x'\|_\infty\}\le\hat\rho/2.\]
Again with \eqref{def:R}, it follows that 
$\hat x\in R(\partial^0_{\check\rho}\check M)$.
\end{proof}

The proof of the following auxiliary result is similar to that of Lemma \ref{lem:v:cover}.

\begin{lemma}\label{lem:v:I:cover}
For all $\hat x\in\Delta_{\hat\rho}$ and $v\in B_{\hat\rho/2}(\hat x)$,
there exists 
\[
\check x\in B_{\hat\rho/2}(\hat x)\cap B_{\check\rho/2}(v)
\cap\Delta_{\check\rho}.
\]
\end{lemma}

\begin{proof}
Consider the point $\check x\in\Delta_{\check\rho}$ given by
\[
\check x_j:=\begin{cases}
\lfloor\tfrac{v_j}{\check\rho}\rfloor\check\rho,
&\tfrac{v_j}{\check\rho}-\lfloor\tfrac{v_j}{\check\rho}\rfloor<\tfrac12,\\
\lfloor\tfrac{v_j}{\check\rho}\rfloor\check\rho,
&\tfrac{v_j}{\check\rho}-\lfloor\tfrac{v_j}{\check\rho}\rfloor=\tfrac12,\ \hat x_j<v_j,\\
\lceil\tfrac{v_j}{\check\rho}\rceil\check\rho,
&\tfrac{v_j}{\check\rho}-\lfloor\tfrac{v_j}{\check\rho}\rfloor=\tfrac12,\ v_j\le\hat x_j,\\
\lceil\tfrac{v_j}{\check\rho}\rceil\check\rho,
&\tfrac{v_j}{\check\rho}-\lfloor\tfrac{v_j}{\check\rho}\rfloor>\tfrac12
\end{cases}
\]
for all $j\in[1,m]$. 
We first check that for all $j\in[1,m]$, we have
\begin{align} 
&|\check x_j-\hat x_j|<\tfrac{\hat\rho}{2}+\tfrac{\check\rho}{2},\label{aim:1a}\\
&|v_j-\check x_j|\le\tfrac{\check\rho}{2}.\label{aim:2a}
\end{align}
Case 1: When $\tfrac{v_j}{\check\rho}-\lfloor\tfrac{v_j}{\check\rho}\rfloor<\tfrac12$, 
then \eqref{aim:1a} follows from 
$\check x_j-\hat x_j=\lfloor\tfrac{v_j}{\check\rho}\rfloor\check\rho-\hat x_j$
and
\[
-\tfrac{\hat\rho}{2}-\tfrac{\check\rho}{2}
\le v_j-\tfrac{\check\rho}{2}-\hat x_j
<\lfloor\tfrac{v_j}{\check\rho}\rfloor\check\rho-\hat x_j
\le v_j-\hat x_j
\le\tfrac{\hat\rho}{2}.
\]
Inequality \eqref{aim:2a} holds because
$|v_j-\check x_j|
=|v_j-\lfloor\tfrac{v_j}{\check\rho}\rfloor\check\rho|
=|\tfrac{v_j}{\check\rho}-\lfloor\tfrac{v_j}{\check\rho}\rfloor|\check\rho
<\tfrac{\check\rho}{2}$.

Case 2: When $\tfrac{v_j}{\check\rho}-\lfloor\tfrac{v_j}{\check\rho}\rfloor=\tfrac12$ 
and $\hat x_j<v_j$, then inequality \eqref{aim:1a} follows from
$\check x_j-\hat x_j
=\lfloor\tfrac{v_j}{\check\rho}\rfloor\check\rho-\hat x_j
=v_j-\hat x_j-\tfrac{\check\rho}{2}$ 
and
\[
-\tfrac{\check\rho}{2}
<v_j-\hat x_j-\tfrac{\check\rho}{2}
<v_j-\hat x_j\le\tfrac{\hat\rho}{2}.
\]
Inequality \eqref{aim:2a} holds because
$|v_j-\check x_j|
=|\tfrac{v_j}{\check\rho}-\lfloor\tfrac{v_j}{\check\rho}\rfloor|\check\rho
=\tfrac{\check\rho}{2}$.

Case 3: When $\tfrac{v_j}{\check\rho}-\lfloor\tfrac{v_j}{\check\rho}\rfloor=\tfrac12$
and $v_j\le\hat x_j$, then inequality \eqref{aim:1a} follows from
$\check x_j-\hat x_j
=\lceil\tfrac{v_j}{\check\rho}\rceil\check\rho-\hat x_j
=v_j+\tfrac{\check\rho}{2}-\hat x_j$
and
\[
-\tfrac{\hat\rho}{2}+\tfrac{\check\rho}{2}
\le v_j+\tfrac{\check\rho}{2}-\hat x_j
\le\tfrac{\check\rho}{2}.
\]
Inequality \eqref{aim:2a} holds because
$|v_j-\check x_j|
=|\tfrac{v_j}{\check\rho}-\lceil\tfrac{v_j}{\check\rho}\rceil|\check\rho
=\tfrac{\check\rho}{2}$.

Case 4: When $\tfrac{v_j}{\check\rho}-\lfloor\tfrac{v_j}{\check\rho}\rfloor>\tfrac12$,
then \eqref{aim:1a} follows from 
$\check x_j-\hat x_j
=\lceil\tfrac{v_j}{\check\rho}\rceil\check\rho-\hat x_j$
and
\[
-\tfrac{\hat\rho}{2}
\le v_j-\hat x_j
\le\lceil\tfrac{v_j}{\check\rho}\rceil\check\rho-\hat x_j
<v_j+\tfrac{\check\rho}{2}-\hat x_j
\le\tfrac{\hat\rho}{2}+\tfrac{\check\rho}{2}.
\]
Inequality \eqref{aim:2a} holds because
$|v_j-\check x_j|
=|\tfrac{v_j}{\check\rho}-\lceil\tfrac{v_j}{\check\rho}\rceil|\check\rho
<\tfrac{\check\rho}{2}$.

Since $\|\check x-\hat x\|_\infty\in\check\rho\N_0$, the statement of the 
lemma follows from inequalities \eqref{aim:1a} and \eqref{aim:2a}.
\end{proof}

The properties of the operator $I$ complement those of the operator $R$.

\begin{theorem}\label{thm:I:properties}
The operator $I:S_{\hat\rho}\to S_{\check\rho}$ satisfies 
$I(\Delta_{\hat\rho})=\Delta_{\check\rho}$.
It has the set-theoretical properties
\begin{align}
&\forall\,\hat M,\hat M'\in S_{\hat\rho}:
&&I(\hat M\cup\hat M')=I(\hat M)\cup I(\hat M'),&
\label{I:union}\\
&\forall\,\hat M,\hat M'\in S_{\hat\rho}:
&&(\hat M\subset\hat M')\Rightarrow(I(\hat M)\subset I(\hat M')),&
\label{I:monotone}\\
&\forall\,\hat M\in S_{\hat\rho},\ 
\forall\,\hat x\in\Delta_{\hat\rho}:
&&I(\hat x+\hat M)=\hat x+I(\hat M),&
\label{I:Minkowski:sum}
\end{align}
the approximation properties
\begin{align}
&\forall\,\hat M\in S_{\hat\rho}:&&\dist_H(I(\hat M),\hat M)\le\hat\rho/2,&
\label{I:Hausdorff}\\
&\forall\,\hat M\in S_{\hat\rho}:
&&I(\hat M)\in V_{\hat\rho}^{\check\rho}(\hat M),&
\label{I:voronoi:overapproximation}\\
&\forall\,\hat M\in S_{\hat\rho},\ 
\forall\,\check M\in V_{\hat\rho}^{\check\rho}(\hat M):
&&I(\hat M)\subset\check M,&
\label{I:minimal:voronoi:overapproximation}
\end{align}
and the topology-related properties
\begin{align}
&\forall\,\hat M\in C_{\hat\rho}:
&&I(\hat M)\in C_{\check\rho},&\label{I:connected}\\
&\forall\,\hat M\in S_{\hat\rho}:
&&\partial^0_{\check\rho}I(\hat M)\subset I(\partial^0_{\hat\rho}\hat M).&
\label{I:new:boundary:in:old:boundary}
\end{align}
\end{theorem}

\begin{proof}
The proofs of the statements \eqref{I:union}, \eqref{I:monotone} and \eqref{I:Minkowski:sum}
are elementary, and it is obvious that $I(\Delta_{\hat\rho})=\Delta_{\check\rho}$.

\medskip

Let $\hat M\in S_{\hat\rho}$.
Since $\hat M\in S_{\check\rho}$, we have $\hat M\subset I(\hat M)$, so the operator 
$I$ is well-defined and we have $\dist(\hat M,I(\hat M))=0$.
We see directly from \eqref{def:I} that 
$\dist(I(\hat M),\hat M)\le\hat\rho/2$.
All in all, we have verified \eqref{I:Hausdorff}.

\medskip 

In view of \eqref{def:I} and \eqref{def:V}, statement \eqref{I:voronoi:overapproximation} 
is equivalent with
\[B_{\hat\rho/2}(\hat M)\subset B_{\check\rho/2}(B_{\hat\rho/2}(\hat M)\cap\Delta_{\check\rho}))
\quad\forall\hat M\in S_{\hat\rho}.\]
To check the above inclusion,
let $\hat M\in S_{\hat\rho}$ and $v\in B_{\hat\rho/2}(\hat M)$.
Then there exists $\hat x\in\hat M$ with $v\in B_{\hat\rho/2}(\hat x)$.
By Lemma \ref{lem:v:I:cover}, there exists 
$\check x\in B_{\hat\rho/2}(\hat x)\cap\Delta_{\check\rho}$
with $v\in B_{\check\rho/2}(\check x)$.
But then, as desired, we have
\[v\in B_{\check\rho/2}(\check x)
\subset B_{\check\rho/2}(B_{\hat\rho/2}(\hat x)\cap\Delta_{\check\rho}) 
\subset B_{\check\rho/2}(B_{\hat\rho/2}(\hat M)\cap\Delta_{\check\rho})).\] 

\medskip

We check \eqref{I:minimal:voronoi:overapproximation}.
Let $\hat M\in S_{\hat\rho}$, let $\check M\in S_{\check\rho}$, and assume that there exists
$\check x\in I(\hat M)\cap\check M^c$.
Then $\check x\in\Delta_{\check\rho}$ and $\check M\subset\Delta_{\check\rho}$ imply 
$\check x\in B_{\check\rho/2}(\check M)^c$.
Since $\check x\in I(\hat M)\subset B_{\hat\rho/2}(\hat M)$, it follows that 
$\check M\notin V_{\hat\rho}^{\check\rho}(\hat M)$.

\medskip

We check \eqref{I:connected}.
Let $\hat M\in C_{\hat\rho}$, and let $\check x\in I(\hat M)$ and $\check x'\in I(\hat M)$.
By \eqref{def:I} there exist $\hat x\in \hat M$ and $\hat x'\in \hat M$ with $\|\hat x-\check x\|_\infty\le{\hat\rho}/2$ 
and $\|\hat x'-\check x'\|_\infty\le{\hat\rho}/2$.
Since $\hat M\in C_{\hat\rho}$, there exists $\hat{p}\in P_{\hat\rho}(\hat x,\hat x')$ with $\hat{p}(\ell)\in \hat M$ for all 
$\ell\in[0,L(\hat{p})]$.
With $\phi$ as in Lemma \ref{lem:grid:line}, we define the paths
\[\check{p}_\ell:=\phi(\,\cdot\,;\hat{p}(\ell),\hat{p}(\ell+1);\check\rho),\quad\ell\in[0,L(\hat{p})),\]
and using concatenation as specified in Definition \ref{def:concatenate},
we define
\[\check{p}:=\phi(\,\cdot\,;\hat x',\check x';\check\rho)\circ \check{p}_{L(\hat{p})-1}\circ \ldots\circ \check{p}_1\circ \check{p}_0\circ 
\phi(\,\cdot\,;\check x,\hat x;\check\rho)
\in P_{\check\rho}(\check x,\check x').\]
By construction, for all $\ell\in[0,L(\hat{p}))$, we have 
$\check{p}_\ell(0)=\hat{p}(\ell)\in \hat M\subset I(\hat M)$ and
$\check{p}_\ell(L(\check{p}_\ell))=\hat{p}(\ell+1)\in \hat M\subset I(\hat M)$.
For all $\ell\in[0,L(\hat{p}))$ and every $\ell'\in(0,L(\check{p}_\ell))$,
statements \eqref{distances:to:endpoints:0} 
and \eqref{distances:to:endpoints:1} imply that
\[\|\check{p}_\ell(\ell')-\hat{p}(\ell)\|_\infty+\|\check{p}_\ell(\ell')-\hat{p}(\ell+1)\|_\infty
=\|\hat{p}(\ell)-\hat{p}(\ell+1)\|_\infty\le{\hat\rho}.\]
As a consequence, we have
\[\check{p}_\ell(\ell')\in B_{{\hat\rho}/2}(\hat{p}(\ell))\cup B_{{\hat\rho}/2}(\hat{p}(\ell+1))\subset B_{{\hat\rho}/2}(\hat M).\]
Since $\check{p}_\ell(\ell')\in\Delta_{\check\rho}$, 
we have $\check{p}_\ell(\ell')\in I(\hat M)$.
Similarly, statement \eqref{distances:to:endpoints:1} 
yields
\[\|\phi(\ell';\check x,\hat x;\check\rho)-\hat x\|_\infty\le\|\check x-\hat x\|_\infty\le{\hat\rho}/2
\quad\forall\,\ell'\in[0,L(\phi(\,\cdot\,;\check x,\hat x;\check\rho)],\]
and since $\phi(\ell';\check x,\hat x;\check\rho)\in\Delta_{\check\rho}$, we find that
$\phi(\ell';\check x,\hat x;\check\rho)\in I(\hat M)$ for every \linebreak
$\ell'\in[0,L(\phi(\,\cdot\,;\check x,\hat x;\check\rho)]$.
The same arguments apply to $\phi(\,\cdot\,;\hat x',\check x';\check\rho)$.
All in all, we have $\check{p}(\ell')\in I(\hat M)$ for all $\ell'\in[0,L(\check{p})]$,
and hence $I(\hat M)\in C_{\check\rho}$.

\medskip

We check \eqref{I:new:boundary:in:old:boundary}.
Let $\check{x}\in\partial^0_{\check\rho}I(\hat{M})$.
By \eqref{def:partial:1} and \eqref{def:I}, 
there exists $\hat{x}\in \hat{M}$ with $\|\hat{x}-\check{x}\|_\infty\le\hat{\rho}/2$.
By \eqref{def:partial:1}, there exists $\check{z}\in I(\hat{M})^c\cap\Delta_{\check\rho}$
with $\|\check{x}-\check{z}\|_\infty=\check\rho$.
By \eqref{def:I}, we have $\dist(\check{z},\hat{M})>\hat{\rho}/2$.
The point $\hat{z}\in\Delta_{\hat{\rho}}$ given by $\hat{z}_j=\rd(\check{z}_j/\hat{\rho})\hat{\rho}$ for all 
$j\in[1,m]$ satisfies $\|\hat{z}-\check{z}\|_\infty\le\hat{\rho}/2$.
Since
\[\dist(\hat{z},\hat{M})\ge\dist(\check{z},\hat{M})-\|\hat{z}-\check{z}\|_\infty>0,\]
we have $\hat{z}\in\hat{M}^c\cap\Delta_{\hat{\rho}}$.
Since $\hat{x}\neq \hat{z}$, since 
\[
\|\hat{x}-\hat{z}\|_\infty
\le\|\hat{x}-\check{x}\|_\infty+\|\check{x}-\check{z}\|_\infty+\|\check{z}-\hat{z}\|_\infty
\le\hat{\rho}+\check\rho<2\hat{\rho}
\]
and since $\|\hat{x}-\hat{z}\|_\infty\in\hat{\rho}\N$, we have $\|\hat{x}-\hat{z}\|_\infty=\hat{\rho}$.
Again by \eqref{def:partial:1}, we have $\hat{x}\in\partial^0_{\hat{\rho}}\hat{M}$.
All in all, we have shown that 
$\check{x}\in B_{\hat{\rho}/2}(\hat{x})\cap\Delta_{\check\rho}\subset I(\partial^0_{\hat{\rho}} \hat{M})$.
\end{proof}

Ideally, the composition $R\circ I:S_{\hat\rho}\to S_{\hat\rho}$
should be the identity.

\begin{theorem}\label{thm:almost:id}
The operators $R:S_{\check{\rho}}\to S_{{\hat{\rho}}}$ and $I:S_{\hat{\rho}}\to S_{\check{\rho}}$ satisfy
\begin{align}
&({\hat{\rho}}/\check{\rho}\in 2\N_1+1)
&&\Rightarrow
&&(\forall\,\hat{M}\in S_{\hat{\rho}}:\ R(I(\hat{M}))=\hat{M}),&
\label{I:R:id}\\
&({\hat{\rho}}/\check{\rho}\in 2\N_1)
&&\Rightarrow 
&&(\forall\,\hat{M}\in S_{\hat{\rho}}:\ R(I(\hat{M}))=B_{\hat{\rho}}(\hat{M})\cap\Delta_{\hat{\rho}}).&
\label{I:R:almost:id}
\end{align}
\end{theorem}

\begin{proof}
We check \eqref{I:R:id}.
Let ${\hat{\rho}}/\check{\rho}\in 2\N_1+1$, let $\hat{M}\in S_{\hat{\rho}}$,
and let $\hat{x}\in R(I(\hat{M}))$.
By \eqref{def:R}, we have $\hat{x}\in\Delta_{\hat{\rho}}$, and there exists $\check{x}\in I(\hat{M})$
with $\|\hat{x}-\check{x}\|_\infty\le{\hat{\rho}}/2$.
By \eqref{def:I}, we have $\check{x}\in\Delta_{\check{\rho}}$, and there exists $\hat{x}'\in \hat{M}$ with
$\|\check{x}-\hat{x}'\|_\infty\le{\hat{\rho}}/2$.
Since $\hat{x}\in\Delta_{\check{\rho}}$ and $\check{x}\in\Delta_{\check{\rho}}$, we have 
$\|\hat{x}-\check{x}\|_\infty\in\check{\rho}\N_0\cap[0,{\hat{\rho}}/2]$, and since ${\hat{\rho}}/2\notin\check{\rho}\N_0$,
we have $\|\hat{x}-\check{x}\|_\infty<{\hat{\rho}}/2$.
Hence
\[\|\hat{x}-\hat{x}'\|_\infty\le\|\hat{x}-\check{x}\|_\infty+\|\check{x}-\hat{x}'\|_\infty<{\hat{\rho}},\]
and since $\hat{x}\in\Delta_{\hat{\rho}}$ and $\hat{x}'\in\Delta_{\hat{\rho}}$, this implies $\hat{x}=\hat{x}'$.
All in all, we see that
\begin{equation}\label{local:5}
R(I(\hat{M}))\subset \hat{M}.
\end{equation}
Conversely, when $\hat{x}\in \hat{M}$, then by \eqref{def:I}, we have $\hat{x}\in I(\hat{M})$,
and by \eqref{def:R}, we have $\hat{x}\in R(I(\hat{M}))$.
Hence
\begin{equation}\label{local:6}
\hat{M}\subset R(I(\hat{M})).
\end{equation}
Combining \eqref{local:5} and \eqref{local:6} yields \eqref{I:R:id}.

\medskip

We check \eqref{I:R:almost:id}.
Let ${\hat{\rho}}/\check{\rho}\in 2\N_1$, and let $\hat{M}\in S_{\hat{\rho}}$.
It follows 
from \eqref{def:R} and \eqref{def:I} that 
\begin{equation}\label{local:7}
R(I(\hat{M}))\subset B_{\hat{\rho}}(\hat{M})\cap\Delta_{\hat{\rho}}.
\end{equation}
Let $\hat{z}\in B_{\hat{\rho}}(\hat{M})\cap\Delta_{\hat{\rho}}$, 
and let $\hat{x}\in \hat{M}$ with $\|\hat{x}-\hat{z}\|_\infty\le{\hat{\rho}}$.
Since ${\hat{\rho}}/\check{\rho}\in 2\N_1$, we have 
$\check{x}:=(\hat{x}+\hat{z})/2\in\Delta_{\check{\rho}}$, and hence
$\check{x}\in B_{{\hat{\rho}}/2}(\hat{x})\cap\Delta_{\check{\rho}}\subset I(\hat{M})$.
It follows that 
$\hat{z}\in B_{{\hat{\rho}}/2}(\check{x})\cap\Delta_{\hat{\rho}}\subset R(I(\hat{M}))$,
and we have shown that
\begin{equation}\label{local:8}
B_{\hat{\rho}}(\hat{M})\cap\Delta_{\hat{\rho}}\subset R(I(\hat{M})).
\end{equation}
Combining \eqref{local:7} and \eqref{local:8} yields \eqref{I:R:almost:id}.
\end{proof}

Property \eqref{I:R:almost:id} could be considered a flaw 
in the construction of the operators $R$ and $I$.
However, it is impossible to improve \eqref{I:R:almost:id}  
without sacrificing the covering properties
\eqref{voronoi:overapproximation} and \eqref{I:voronoi:overapproximation}.

\begin{theorem}\label{thm:no:better:covering:operators}
When $\hat\rho/\check{\rho}\in 2\N$, there is no pair 
of operators $R':S_{\check{\rho}}\to S_{\hat\rho}$
and $I':S_{\hat\rho}\to S_{\check{\rho}}$ such that 
\begin{align}
&\forall\,\check{M}\in S_{\check{\rho}}:&&R'(\check{M})\in V_{\check{\rho}}^{\hat\rho}(\check{M}),&\label{R:covers}\\
&\forall\,\hat{M}\in S_{\hat\rho}:&&I'(\hat{M})\in V_{\hat\rho}^{\check{\rho}}(\hat{M}),&\label{I:covers}\\
&\forall\,\hat{M}\in S_{\hat\rho}:&&R'(I'(\hat{M}))=\hat{M}.&\label{R:I:identity}
\end{align}
\end{theorem}

\begin{proof}
Assume that $R'$ and $I'$ satisfy \eqref{R:covers} and \eqref{I:covers},
and let $\hat{M}\in S_{\hat\rho}$.
By 
\eqref{I:covers} and \eqref{I:minimal:voronoi:overapproximation}, we have
$I(\hat{M})\subset I'(\hat{M})$.
By statement \eqref{R:monotone}, it follows that 
$R(I(\hat{M}))\subset R(I'(\hat{M}))$.
By \eqref{R:covers} and \eqref{minimal:voronoi:overapproximation},
we have $R(I'(\hat{M}))\subset R'(I'(\hat{M}))$.
Since $\hat{M}\in S_{\hat\rho}$, and because of \eqref{I:R:almost:id}, we obtain
\[\hat{M}\subsetneq B_{\hat\rho}(\hat{M})\cap\Delta_{\hat\rho}
=R(I(\hat{M}))
\subset R'(I'(\hat{M})),\]
and hence statement \eqref{R:I:identity} is invalid.
\end{proof}

While $R$ selects for every $\check M\in S_{\check\rho}$ 
an element $R(\check M)\in A_{\hat\rho}(\check M)$, i.e.\ a best approximation
to $\check M$ in $S_{\hat\rho}$, this approach leads to undesirable outcomes
for the operator $I$.

\begin{example}\label{ex:I}
For $\Delta_{\hat\rho}\in S_{\hat\rho}$, we also have $\Delta_{\hat\rho}\in S_{\check\rho}$ and hence
\[\Delta_{\hat\rho}=\argmin_{\check M\in S_{\check\rho}}\dist_H(\check M,\Delta_{\hat\rho}).\]
However, we have $\partial^0_{\hat\rho}\Delta_{\hat\rho}=\emptyset$ and
$\partial^0_{\check\rho}\Delta_{\hat\rho}=\Delta_{\hat\rho}$, which means that the 
best approximation of $\Delta_{\hat\rho}$ in $S_{\check\rho}$ violates 
condition \eqref{I:new:boundary:in:old:boundary} 
in an extreme way.
It also fails to meet conditions \eqref{I:voronoi:overapproximation} and 
\eqref{I:connected}. 
\end{example}


\section{Algorithms evaluating lifted operators}
\label{implementations}

Algorithms \ref{Alg:coarsen} and \ref{Alg:refine} are optimized for readability
in conjunction with the theory presented so far, and not for performance.
Note that the distances in Algorithm \ref{Alg:coarsen} are always small and hence 
can be determined by querying a small number of points.
A small computational example is provided in Figure \ref{fig:computational:example}.

\medskip

Given trivial input, our algorithms must compute trivial output.
Recall the operators $\partial R$ and $\partial I$ from \eqref{def:dR} and \eqref{def:dI}.

\begin{lemma}\label{lem:all:trivial}
We have $\partial R(\emptyset,\emptyset)=(\emptyset,\emptyset)$
and $\partial I(\emptyset,\emptyset)=(\emptyset,\emptyset)$.
\end{lemma}

\begin{proof}
By Theorem \ref{thm:F:well:defined}, we have 
$\tr_{\check\rho}^{-1}(\emptyset,\emptyset)=\Delta_{\check\rho}$ and
$\tr_{\hat\rho}^{-1}(\emptyset,\emptyset)=\Delta_{\hat\rho}$.
By Theorems \ref{thm:R:properties} and \ref{thm:I:properties}, 
we have $R(\Delta_{\check\rho})=\Delta_{\hat\rho}$ and 
$I(\Delta_{\hat\rho})=\Delta_{\check{\rho}}$.
By \eqref{Delta:rho:equiv:empty:partial}, we have 
$\tr_{\hat\rho}(\Delta_{\hat\rho})=(\emptyset,\emptyset)$
and $\tr_{\check{\rho}}(\Delta_{\check{\rho}})=(\emptyset,\emptyset)$,
so by \eqref{def:dR} and \eqref{def:dI}, the desired statements hold.
\end{proof}

\begin{algorithm}
\KwIn{$(\check{D}_0,\check{D}_1)\in\bd_{\check\rho}$}
\KwOut{$(\hat{D}_0,\hat{D}_1)\in\bd_{\hat\rho}$}
$\hat{H}_0\gets\emptyset$\\
$\hat{H}_1\gets\emptyset$\\
\For{$\hat{x}\in B_{3\hat\rho/2}(\check{D}_0)\cap\Delta_{\hat\rho}$}{
  \uIf{$\dist(\hat{x},\check{D}_0)\le\hat\rho/2$}{$\hat{H}_0\gets \hat{H}_0\cup\{\hat{x}\}$}
  \ElseIf{$\dist(\hat{x},\check{D}_1)\le\dist(\hat{x},\check{D}_0)$}{$\hat{H}_1\gets \hat{H}_1\cup\{\hat{x}\}$}
}
$\hat{D}_1\gets\{\hat{x}\in \hat{H}_1:\ \dist(\hat{x},\hat{H}_0)=\hat\rho\}$\\
$\hat{D}_0\gets\{\hat{x}\in \hat{H}_0:\ \dist(\hat{x},\hat{H}_1)=\hat\rho\}$
\caption{Evaluates $\partial R$ from \eqref{def:dR}.}
\label{Alg:coarsen}
\end{algorithm}

We prove that Algorithm \ref{Alg:coarsen} evaluates $\partial R$
correctly. 

\begin{theorem}\label{Thm:CoarsenWorks}
Let $(\check{D}_0,\check{D}_1)\in\bd_{\check\rho}$ and $(\hat{D}_0,\hat{D}_1)=\partial R(\check{D}_0,\check{D}_1)$.
Then the sets $\hat{H}_0\subset\Delta_{{\hat\rho}}$ and $\hat{H}_1\subset\Delta_{{\hat\rho}}$
generated by Algorithm \ref{Alg:coarsen} satisfy
\begin{align}    
&\hat{D}_0=\{\hat{x}\in \hat{H}_0:\ \dist(\hat{x},\hat{H}_1)={\hat\rho}\},
\label{Coarse:eq:AlgOutput0}\\
&\hat{D}_1=\{\hat{x}\in \hat{H}_1:\ \dist(\hat{x},\hat{H}_0)={\hat\rho}\}.
\label{Coarse:eq:AlgOutput1}
\end{align}
\end{theorem} 

\begin{proof}
If $(\check{D}_0,\check{D}_1)=(\emptyset,\emptyset)$, then the loop in Algorithm \ref{Alg:coarsen}
is void, and thus we have $(\hat{H}_0,\hat{H}_1)=(\emptyset,\emptyset)$.
By Lemma \ref{lem:all:trivial},
we also have $(\hat{D}_0,\hat{D}_1)=(\emptyset,\emptyset)$, and 
statements \eqref{Coarse:eq:AlgOutput0} and \eqref{Coarse:eq:AlgOutput1} hold.

\medskip

From now on let $(\check{D}_0,\check{D}_1)\in\bd_{\check\rho}^-$.
Inspecting Algorithm \ref{Alg:coarsen} reveals that
\begin{align}
&\hat{H}_0=B_{{\hat\rho}/2}(\check{D}_0)\cap\Delta_{{\hat\rho}},
\label{H0:as:in:Alg:1}\\
&\hat{H}_1=\{\hat{x}\in B_{3{\hat\rho}/2}(\check{D}_0)\cap B_{{\hat\rho}/2}(\check{D}_0)^c\cap\Delta_{{\hat\rho}}:
\dist(\hat{x},\check{D}_1)\le\dist(\hat{x},\check{D}_0)\}.
\label{H1:as:in:Alg:1}
\end{align}
Let $\check{M}:=\tr_{\check\rho}^{-1}(\check{D}_0,\check{D}_1)$.
Then Theorem \ref{thm:F:well:defined} and \eqref{def:dR} 
yield
\begin{equation}\label{this}
\check{D}_0=\partial^0_{\check\rho} \check{M},\quad
\check{D}_1=\partial^1_{\check\rho} \check{M},\quad
\hat{D}_0=\partial^0_{{\hat\rho}}R(\check{M}),\quad
\hat{D}_1=\partial^1_{{\hat\rho}}R(\check{M}).
\end{equation}
We first argue that
\begin{equation}\label{Coarse:eq:H0SetIneqs}
\partial^0_{{\hat\rho}}R(\check{M})
\subset \hat{H}_0\subset R(\check{M}).
\end{equation}
Using \eqref{new:boundary:in:old:boundary}, \eqref{def:R}, \eqref{this} 
and \eqref{H0:as:in:Alg:1}, we see that
\[
\partial^0_{{\hat\rho}}R(\check{M})
\subset R(\partial^0_{\check\rho} \check{M})
=B_{{\hat\rho}/2}(\partial^0_{\check\rho} \check{M})\cap\Delta_{{\hat\rho}}
=B_{{\hat\rho}/2}(\check{D}_0)\cap\Delta_{{\hat\rho}}
=\hat{H}_0,
\]
so the first inclusion in \eqref{Coarse:eq:H0SetIneqs} holds.
Using \eqref{H0:as:in:Alg:1}, \eqref{this} and \eqref{def:R}, we obtain,
partly repeating the above computation, that
\[
\hat{H}_0=B_{{\hat\rho}/2}(\check{D}_0)\cap\Delta_{{\hat\rho}}
=B_{{\hat\rho}/2}(\partial^0_{\check\rho} \check{M})\cap\Delta_{{\hat\rho}}
\subset B_{{\hat\rho}/2}(\check{M})\cap\Delta_{{\hat\rho}}
=R(\check{M}),
\]
and hence the second inclusion in \eqref{Coarse:eq:H0SetIneqs} holds.

\medskip

In the following, we argue that
\begin{equation}\label{Coarse:eq:H1SetIneqs}
\partial^1_{{\hat\rho}}R(\check{M})
\subset \hat{H}_1\subset R(\check{M})^c\cap\Delta_{{\hat\rho}}.
\end{equation}
Since $\Delta_{{\hat\rho}}\subset\Delta_{\check\rho}$, statements \eqref{coarse:eq:xdists} 
and \eqref{opposite:coarse:eq:xdists} allow us in conjunction with \eqref{this} 
to represent \eqref{H1:as:in:Alg:1} in the form
\begin{equation}\label{better:representation:H1}
\hat{H}_1=B_{3{\hat\rho}/2}(\partial^0_{\check\rho} \check{M})\cap B_{{\hat\rho}/2}(\partial^0_{\check\rho} \check{M})^c
\cap \check{M}^c\cap\Delta_{{\hat\rho}}.
\end{equation}
Now let $\hat{z}\in\partial^1_{{\hat\rho}}R(\check{M})$.
It follows from the triangle inequality, from \eqref{def:partial:1},
from \eqref{new:boundary:in:old:boundary} and from \eqref{def:R} that
\begin{equation}\label{check:H1:1}\begin{aligned}
\dist(\hat{z},\partial^0_{\check\rho} \check{M})
&\le\dist(\partial^1_{{\hat\rho}}R(\check{M}),\partial^0_{{\hat\rho}}R(\check{M}))
+\dist(\partial^0_{{\hat\rho}}R(\check{M}),\partial^0_{\check\rho} \check{M})\\
&\le{\hat\rho}+\dist(R(\partial^0_{\check\rho} \check{M}),\partial^0_{\check\rho} \check{M})
\le 3{\hat\rho}/2.
\end{aligned}
\end{equation}
Also, since $\partial^0_{\check\rho} \check{M}\subset \check{M}$, since $\hat{z}\in R(\check{M})^c\cap\Delta_{{\hat\rho}}$
and by \eqref{def:R},
it follows 
that
\begin{equation}\label{check:H1:2}
\dist(\hat{z},\partial_{\check\rho}^0\check{M})\ge\dist(\hat{z},\check{M})>{\hat\rho}/2.
\end{equation}
Finally, by triangle inequality, since $\hat{z}\in R(\check{M})^c\cap\Delta_{{\hat\rho}}$ 
and by \eqref{eq:GammaHausdorff}, we have
\[\dist(\hat{z},\check{M})
\ge\dist(\hat{z},R(\check{M}))-\dist(\check{M},R(\check{M}))
\ge{\hat\rho}-{\hat\rho}/2\ge{\check\rho},\]
and hence
\begin{equation}\label{check:H1:3}
\hat{z}\in \check{M}^c\cap\Delta_{{\hat\rho}}. 
\end{equation}
Comparing statements \eqref{check:H1:1}, \eqref{check:H1:2} and \eqref{check:H1:3}
with representation \eqref{better:representation:H1} shows that $\hat{z}\in \hat{H}_1$. 
Hence the first inclusion in \eqref{Coarse:eq:H1SetIneqs} holds.

\medskip

We check the second inclusion in \eqref{Coarse:eq:H1SetIneqs}.
Let $\hat{z}\in \hat{H}_1$.
By \eqref{better:representation:H1}, we have 
$\hat{z}\in \check{M}^c\cap\Delta_{{\hat\rho}}\subset \check{M}^c\cap\Delta_{\check\rho}$.
Using \eqref{coarse:eq:xnotinMdist} and \eqref{better:representation:H1}, we conclude that
\[\dist(\hat{z},\check{M})=\dist(\hat{z},\partial_{\check\rho}^0\check{M})>{\hat\rho}/2.\]
Because of \eqref{eq:GammaHausdorff}, it follows that $\hat{z}\notin R(\check{M})$, 
and since $\hat{z}\in\Delta_{{\hat\rho}}$, the second inclusion in statement 
\eqref{Coarse:eq:H1SetIneqs} holds as well.

\medskip

Now statements \eqref{Coarse:eq:AlgOutput0} and \eqref{Coarse:eq:AlgOutput1}
follow from \eqref{this}, from
Lemma \ref{prop:bdrycleaning} (with ${\hat\rho}$ and $R(\check{M})$
in lieu of $\rho$ and $M$) and from statements 
\eqref{Coarse:eq:H0SetIneqs} and \eqref{Coarse:eq:H1SetIneqs}.
\end{proof}

We prove that Algorithm \ref{Alg:refine} evaluates $\partial I$ correctly. 
An inspection of the algorithm shows that it computes the unions in equations
\eqref{fine:output:0} and \eqref{fine:output:1}.

\begin{algorithm}
\KwIn{$(\hat{D}_0,\hat{D}_1)\in\bd_{\hat{\rho}}$}
\KwOut{$(\check{D}_0,\check{D}_1)\in\bd_{\check{\rho}}$}
$\check{D}_0\gets\emptyset$\\
$\check{D}_1\gets\emptyset$\\
\For{$\hat{z}\in \hat{D}_1$}{
  \For{$\hat{x}\in \hat{D}_0\cap B_{\hat{\rho}}(\hat{z})$}{
    $\check{D}_0\gets \check{D}_0\cup(B_{{\hat{\rho}}/2}(\hat{x})\cap B_{({\hat{\rho}}+\check{\rho})/2}(\hat{z})\cap\Delta_{\check{\rho}})$\\
    $\check{D}_1\gets \check{D}_1\cup\big(B_{{\hat{\rho}}/2}(\hat{z})\cap B_{{\hat{\rho}}/2+\check{\rho}}(\hat{x})\cap\big(\cap_{\hat{x}'\in \hat{D}_0\cap B_{\hat{\rho}}(\hat{z})}B_{{\hat{\rho}}/2}(\hat{x}')^c\big)\cap\Delta_{\check{\rho}}\big)$  
  }
}
\caption{Evaluates $\partial I$ from \eqref{def:dI}.}
\label{Alg:refine}
\end{algorithm}

\begin{theorem}\label{thm:refining:works}
Let $(\hat{D}_0,\hat{D}_1)\in\bd_{\hat\rho}$ and $(\check{D}_0,\check{D}_1)=\partial I(\hat{D}_0,\hat{D}_1)$.
Then 
\begin{align}    
&\check{D}_0
=\bigcup_{\hat{x}\in \hat{D}_0}\bigcup_{\hat{z}\in \hat{D}_1\cap B_{\hat\rho}(\hat{x})}
\Big(B_{{\hat\rho}/2}(\hat{x})\cap B_{({\hat\rho}+\check{\rho})/2}(\hat{z})\cap\Delta_{\check{\rho}}\Big),
\label{fine:output:0}\\
&\check{D}_1=\bigcup_{\hat{z}\in \hat{D}_1}\bigcup_{\hat{x}\in \hat{D}_0\cap B_{\hat\rho}(\hat{z})}
\Big(B_{{\hat\rho}/2}(\hat{z})\cap B_{{\hat\rho}/2+\check{\rho}}(\hat{x})
\cap\big(\!\!\!\!\!\!\bigcap_{\hat{x}'\in \hat{D}_0\cap B_{\hat\rho}(\hat{z})}\!\!\!\!\!\!B_{{\hat\rho}/2}(\hat{x}')^c\big)
\cap\Delta_{\check{\rho}}\Big).
\label{fine:output:1}
\end{align}
\end{theorem} 

\begin{proof}
If $(\hat{D}_0,\hat{D}_1)=(\emptyset,\emptyset)$, 
then Lemma \ref{lem:all:trivial} yields $(\check{D}_0,\check{D}_1)=(\emptyset,\emptyset)$,
and the unions in equations \eqref{fine:output:0} and \eqref{fine:output:1} are empty.
Hence \eqref{fine:output:0} and \eqref{fine:output:1} hold.

\medskip

From now on let $(\hat{D}_0,\hat{D}_1)\in\bd_{\hat\rho}^-$.
Let $\hat{M}:=\tr_{\hat\rho}^{-1}(\hat{D}_0,\hat{D}_1)$.
Then Theorem \ref{thm:F:well:defined} and \eqref{def:I} yield
\begin{equation}\label{that}
\hat{D}_0=\partial^0_{\hat\rho} \hat{M},\quad
\hat{D}_1=\partial^1_{\hat\rho} \hat{M},\quad
\check{D}_0=\partial^0_{\check{\rho}}I(\hat{M}),\quad
\check{D}_1=\partial^1_{\check{\rho}}I(\hat{M}).
\end{equation}

\medskip

We check \eqref{fine:output:0}.
Let $\check{x}\in \check{D}_0$.
By \eqref{that}, we have $\check{x}\in\partial^0_{\check{\rho}}I(\hat{M})$.
By \eqref{I:new:boundary:in:old:boundary}, \eqref{def:I} and \eqref{that},
there exists $\hat{x}\in\partial^0_{\hat\rho} \hat{M}=\hat{D}_0$ with 
\begin{equation}\label{loc:3}
\check{x}\in B_{{\hat\rho}/2}(\hat{x})\cap\Delta_{\check{\rho}}.
\end{equation}
By \eqref{def:partial:1}, there exists $\check{z}\in I(\hat{M})^c\cap\Delta_{\check{\rho}}$
with 
\begin{equation}\label{loc:8}
\|\check{x}-\check{z}\|_\infty=\check{\rho}.
\end{equation}
In view of \eqref{def:I}, we have 
\begin{equation}\label{loc:1}
\dist(\check{z},\hat{M})>{\hat\rho}/2.
\end{equation}
By \eqref{loc:8} and by Lemma \ref{lem:v:cover} with $k=1$ and with $\check{z}$ in lieu of $v$, 
there exists a point 
\begin{equation}\label{loc:2}
\hat{z}\in B_{(\hat\rho+\check{\rho})/2}(\check{x})\cap B_{\hat\rho/2}(\check{z})\cap\Delta_{\hat\rho}.
\end{equation}
Combining \eqref{loc:1} and \eqref{loc:2} yields $\hat{z}\in \hat{M}^c\cap\Delta_{\hat\rho}$, 
and combining \eqref{loc:3} with \eqref{loc:2} yields $\|\hat{x}-\hat{z}\|_\infty\le{\hat\rho}+\check{\rho}/2$.
Since $\|\hat{x}-\hat{z}\|_\infty\in{\hat\rho}\N_0$, we conclude that $\|\hat{x}-\hat{z}\|_\infty\le{\hat\rho}$.
In particular, we have 
\begin{equation}\label{local:15}
\hat{z}\in\partial^1_{\hat\rho}\hat{M}\cap B_{\hat\rho}(\hat{x})
=\hat{D}_1\cap B_{\hat\rho}(\hat{x}).
\end{equation}
Using \eqref{loc:3}, \eqref{loc:2} and \eqref{local:15}, we conclude that
\[\check{x}\in\cup_{\hat{x}\in \hat{D}_0}\cup_{\hat{z}\in \hat{D}_1\cap B_{\hat\rho}(\hat{x})}
(B_{{\hat\rho}/2}(\hat{x})\cap B_{({\hat\rho}+\check{\rho})/2}(\hat{z})\cap\Delta_{\check{\rho}}).\]

\medskip

Conversely, let $\hat{x}\in \hat{D}_0$, let $\hat{z}\in \hat{D}_1\cap B_{\hat\rho}(\hat{x})$, and let
\begin{equation}\label{local:9}
\check{x}\in B_{{\hat\rho}/2}(\hat{x})\cap B_{({\hat\rho}+\check{\rho})/2}(\hat{z})\cap\Delta_{\check{\rho}}.
\end{equation}
By \eqref{that}, we have $\hat{x}\in\partial^0_{\hat\rho} \hat{M}$ and $\hat{z}\in\partial^1_{\hat\rho} \hat{M}\cap B_{\hat\rho}(\hat{x})$,
and by \eqref{def:I} and \eqref{local:9}, we have 
\begin{equation}\label{local:13}
\check{x}\in I(\hat{M}).
\end{equation}
Again in view of \eqref{local:9}, the point $\check{z}\in\Delta_{\check{\rho}}$ given by
\begin{equation*}
\check{z}_j:=\begin{cases}
\check{x}_j+\check{\rho},&\check{x}_j\in[\hat{z}_j-{\hat\rho}/2-\check{\rho}/2,\hat{z}_j-{\hat\rho}/2],\\
\check{x}_j,&\check{x}_j\in(\hat{z}_j-{\hat\rho}/2,\hat{z}_j+{\hat\rho}/2),\\
\check{x}_j-\check{\rho},&\check{x}_j\in[\hat{z}_j+{\hat\rho}/2,\hat{z}_j+{\hat\rho}/2+\check{\rho}/2]
\end{cases}
\end{equation*}
for $j\in[1,m]$ is well-defined, and we have 
\begin{equation}\label{local:11}
\|\check{x}-\check{z}\|_\infty\le\check{\rho}
\end{equation}
as well as $\check{z}_j\in(\hat{z}_j-{\hat\rho}/2,\hat{z}_j+{\hat\rho}/2)$ for all $j\in[1,m]$, 
and thus $\|\hat{z}-\check{z}\|_\infty<{\hat\rho}/2$.
Since $\hat{z}\in\partial^1_{\hat\rho} \hat{M}$, 
we see, using \eqref{dk:is:dist:k}, that
\[{\hat\rho}=\dist(\hat{z},\hat{M})\le\|\hat{z}-\check{z}\|_\infty+\dist(\check{z},\hat{M})<{\hat\rho}/2+\dist(\check{z},\hat{M}),\]
which shows that $\dist(\check{z},\hat{M})>{\hat\rho}/2$, and hence, in view of \eqref{def:I} that
\begin{equation}\label{local:12}
\check{z}\in I(\hat{M})^c\cap\Delta_{\check{\rho}}.
\end{equation}
In view of \eqref{def:partial:1},
combining \eqref{local:13} with \eqref{local:11} and \eqref{local:12} yields 
$\check{x}\in\partial^0_{\hat\rho} I(\hat{M})$, and by \eqref{that},
we have $\check{x}\in \check{D}_0$.
All in all, we have verified statement \eqref{fine:output:0}.

\medskip

We check statement \eqref{fine:output:1}.
Let $\check{z}\in \check{D}_1$.
By \eqref{that}, we have $\check{z}\in\partial_{\check{\rho}}^1I(\hat{M})$.
Hence we have $\check{z}\in\Delta_{\check{\rho}}$,  
we see with \eqref{def:partial:3} and \eqref{def:I} that 
\eqref{loc:1} holds,
and by Lemma \ref{neighbour:in:boundary}, there exists 
$\check{x}\in\partial_{\check{\rho}}^0 I(\hat{M})$ with 
\eqref{loc:8}.
By \eqref{I:new:boundary:in:old:boundary} and \eqref{that}, there exists 
\begin{equation}\label{local:16}
\hat{x}\in\partial^0_{\hat\rho} \hat{M}=\hat{D}_0
\end{equation} 
with $\check{x}\in I(\hat{x})$.
By \eqref{def:I}, we have \eqref{loc:3}, and thus, with \eqref{loc:8}, that
\[\|\check{z}-\hat{x}\|_\infty\le\|\check{z}-\check{x}\|_\infty+\|\check{x}-\hat{x}\|_\infty\le{\hat\rho}/2+\check{\rho}.\]
We combine this with \eqref{loc:1} and summarize
\begin{equation}\label{loc:7}
\check{z}\in B_{{\hat\rho}/2+\check{\rho}}(\hat{x})\cap\big(\cap_{\hat{x}'\in \hat{D}_0}B_{{\hat\rho}/2}(\hat{x}')^c\big)\cap\Delta_{\check{\rho}}.
\end{equation}
Again by \eqref{loc:8} and by Lemma \ref{lem:v:cover} with $k=1$, 
there exists $\hat{z}\in\Delta_{\hat\rho}$ with \eqref{loc:2}, and 
\eqref{local:15}
holds for the same reasons as before.
From \eqref{loc:2}, \eqref{local:15}, \eqref{local:16} and \eqref{loc:7}, we obtain that
\[\check{z}\in\cup_{\hat{x}\in\hat{D}_0}\cup_{\hat{z}\in\hat{D}_1\cap B_{\hat\rho}(\hat{x})}
\big(B_{{\hat\rho}/2}(\hat{z})\cap B_{{\hat\rho}/2+\check{\rho}}(\hat{x})\cap
\big(\cap_{\hat{x}'\in \hat{D}_0}B_{{\hat\rho}/2}(\hat{x}')^c\big)\cap\Delta_{\check{\rho}}\big).\]
Since 
\[(\cap_{\hat{x}'\in \hat{D}_0}B_{{\hat\rho}/2}(\hat{x}')^c)
\subset(\cap_{\hat{x}'\in \hat{D}_0\cap B_{\hat\rho}(\hat{z})}B_{{\hat\rho}/2}(\hat{x}')^c)\quad
\forall\,\hat{z}\in\Delta_{\hat\rho},\]
it follows that
\[\check{z}\in\cup_{\hat{z}\in \hat{D}_1}\cup_{\hat{x}\in \hat{D}_0\cap B_{\hat\rho}(\hat{z})}
\big(B_{{\hat\rho}/2}(\hat{z})\cap B_{{\hat\rho}/2+\check{\rho}}(\hat{x})\cap
\big(\cap_{\hat{x}'\in \hat{D}_0\cap B_{\hat\rho}(\hat{z})}B_{{\hat\rho}/2}(\hat{x}')^c\big)
\cap\Delta_{\check{\rho}}\big).\]

\medskip

Conversely, let $\hat{z}\in \hat{D}_1$, let $\hat{x}\in \hat{D}_0\cap B_{\hat\rho}(\hat{z})$, and let 
\begin{equation}\label{loc:4}
\check{z}\in B_{{\hat\rho}/2}(\hat{z})\cap B_{{\hat\rho}/2+\check{\rho}}(\hat{x})
\cap\big(\cap_{\hat{x}'\in \hat{D}_0\cap B_{\hat\rho}(\hat{z})}B_{{\hat\rho}/2}(\hat{x}')^c\big)
\cap \Delta_{\check{\rho}}.
\end{equation}
By \eqref{that}, we have $\hat{x}\in\partial^0_{\hat\rho} \hat{M}$ and $\hat{z}\in\partial^1_{\hat\rho} \hat{M}$. 
We claim that
\begin{equation}\label{loc:5}
\check{z}\in I(\hat{M})^c\cap\Delta_{\check{\rho}}.
\end{equation}
Assume that statement \eqref{loc:5} is false.
Then $\check{z}\in I(\hat{M})$, and  by \eqref{def:I}, there exists $\hat{x}'\in \hat{M}$ 
with $\check{z}\in B_{{\hat\rho}/2}(\hat{x}')$.
If $\|\hat{x}'-\hat{z}\|_\infty>{\hat\rho}$, then \eqref{loc:4} yields
\[{\hat\rho}<\|\hat{x}'-\hat{z}\|_\infty
\le\|\hat{x}'-\check{z}\|_\infty+\|\check{z}-\hat{z}\|_\infty
\le\|\hat{x}'-\check{z}\|_\infty+{\hat\rho}/2,\]
and hence the contradiction $\|\hat{x}'-\check{z}\|_\infty>{\hat\rho}/2$.
If $\|\hat{x}'-\hat{z}\|_\infty\le{\hat\rho}$, then Lemma \ref{neighbour:in:boundary}, 
with $\hat{x}'\in \hat{M}$ and $\hat{z}\in\partial^1_{\hat\rho} \hat{M}$,
yields $\hat{x}'\in\partial^0_{\hat\rho} \hat{M}\cap B_{\hat\rho}(\hat{z})$.
By \eqref{that}, it follows that $\hat{x}'\in \hat{D}_0\cap B_{\hat\rho}(\hat{z})$, which, with $\check{z}\in B_{{\hat\rho}/2}(\hat{x}')$,
contradicts \eqref{loc:4}.
All in all, statement \eqref{loc:5} holds.

Let $\phi(\,\cdot\,)=\phi(\,\cdot\,;\hat{x},\check{z};\check{\rho})$ be the path from Lemma \ref{lem:grid:line}.
By \eqref{loc:4} and since $\hat{x}\in \hat{D}_0\cap B_{\hat\rho}(\hat{z})$,
we have $\|\check{z}-\hat{x}\|_\infty>{\hat\rho}/2\ge\check{\rho}$, and hence that $L(\phi)>1$.
Consider 
$\check{x}:=\phi(L(\phi)-1)$.
By \eqref{endpoints:1} and \eqref{step:size:rho}, we have
\begin{equation}\label{loc:6}
\|\check{z}-\check{x}\|_\infty=\check{\rho},
\end{equation}
and by \eqref{distances:to:endpoints:0}, by \eqref{endpoints:1} and 
\eqref{distances:to:endpoints:0}, and by recalling $\check{z}\in B_{{\hat\rho}/2+\check{\rho}}(\hat{x})$ 
from \eqref{loc:4}, we see that
\begin{equation}\label{loc:9}
\|\check{x}-\hat{x}\|_\infty=\|\check{z}-\hat{x}\|_\infty-\check{\rho}\le{\hat\rho}/2.
\end{equation}
Now $\hat{x}\in \hat{M}$ and \eqref{loc:9} yield $\check{x}\in I(\hat{M})$, 
and with \eqref{loc:5}, \eqref{loc:6} and \eqref{that}, it follows that
$\check{z}\in\partial^1_{\check{\rho}} I(\hat{M})=\check{D}_1$.
All in all, we have verified \eqref{fine:output:1}.
\end{proof}

\begin{sidewaysfigure*}[p]
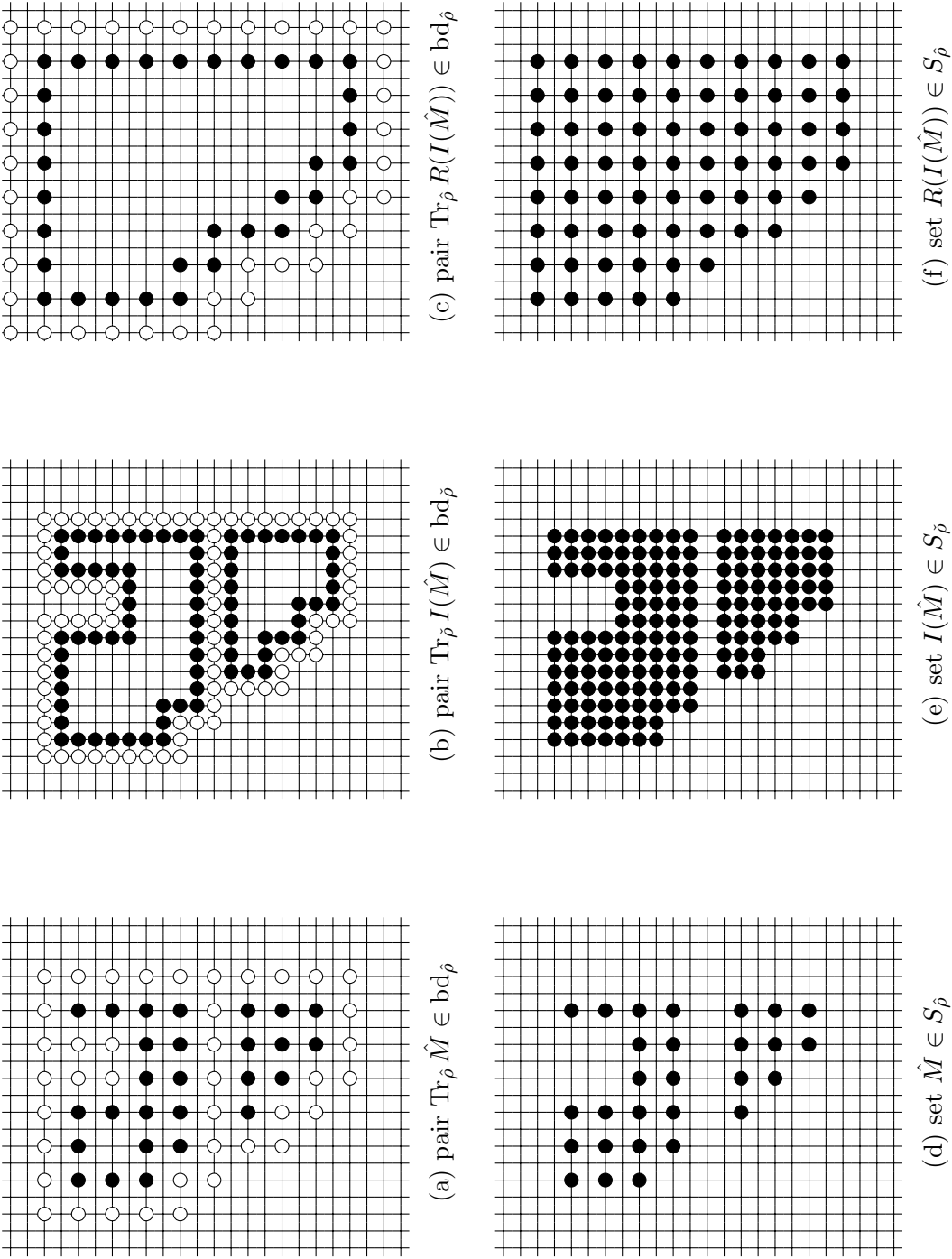

\centering
\begin{subfigure}[t]{0.3\textwidth}
\centering
\showgoboard{
--------------------,
--------------------,
--1-1-1-1-1-1-1-1---,
--------------------,
--1-0-0-0-1-1-0-1---,
--------------------,
--1-0---0-1-1-0-1---,
--------------------,
--1-0-0-0-0-0-0-1---,
--------------------,
--1-1-0-0-0-0-0-1---,
--------------------,
----1-1-1-1-1-1-1---,
--------------------,
------1-0-0-0-0-1---,
--------------------,
------1-1-0-0-0-1---,
--------------------,
--------1-1-0-0-1---,
--------------------,
----------1-1-1-1---,
--------------------,
--------------------,
--------------------}
\caption{pair $\tr_{\hat\rho}\hat M\in\bd_{\hat\rho}$}
\end{subfigure}
\hfill
\begin{subfigure}[t]{0.3\textwidth}
\centering
\showgoboard{
--------------------,
--------------------,
--111111111-11111---,
--100000001-10001---,
--10-----01-10-01---,
--10-----01-10-01---,
--10-----01110-01---,
--10-----00000-01---,
--10-----------01---,
--1000---------01---,
--1110---------01---,
----1000000000001---,
----1111111111111---,
------10000000001---,
------10-------01---,
------1000-----01---,
------1110-----01---,
--------1000---01---,
--------1110---01---,
----------1000001---,
----------1111111---,
--------------------,
--------------------,
--------------------}
\caption{pair $\tr_{\check\rho}I(\hat M)\in\bd_{\check\rho}$}
\end{subfigure}
\hfill
\begin{subfigure}[t]{0.3\textwidth}
\centering
\showgoboard{
1-1-1-1-1-1-1-1-1-1-,
--------------------,
1-0-0-0-0-0-0-0-0-1-,
--------------------,
1-0-------------0-1-,
--------------------,
1-0-------------0-1-,
--------------------,
1-0-------------0-1-,
--------------------,
1-0-0-----------0-1-,
--------------------,
1-1-0-0---------0-1-,
--------------------,
--1-1-0---------0-1-,
--------------------,
----1-0-0-------0-1-,
--------------------,
----1-1-0-0-----0-1-,
--------------------,
------1-1-0-0-0-0-1-,
--------------------,
--------1-1-1-1-1-1-,
--------------------}
\caption{pair $\tr_{\hat\rho}R(I(\hat M))\in\bd_{\hat\rho}$}
\end{subfigure}
\\[0.5cm]
\begin{subfigure}[t]{0.3\textwidth}
\centering
\showgoboard{
--------------------,
--------------------,
--------------------,
--------------------,
----0-0-0-----0-----,
--------------------,
----0-0-0-----0-----,
--------------------,
----0-0-0-0-0-0-----,
--------------------,
------0-0-0-0-0-----,
--------------------,
--------------------,
--------------------,
--------0-0-0-0-----,
--------------------,
----------0-0-0-----,
--------------------,
------------0-0-----,
--------------------,
--------------------,
--------------------,
--------------------,
--------------------}
\caption{set $\hat M\in S_{\hat\rho}$}
\end{subfigure}
\hfill
\begin{subfigure}[t]{0.3\textwidth}
\centering
\showgoboard{
--------------------,
--------------------,
--------------------,
---0000000---000----,
---0000000---000----,
---0000000---000----,
---0000000---000----,
---0000000000000----,
---0000000000000----,
---0000000000000----,
-----00000000000----,
-----00000000000----,
--------------------,
-------000000000----,
-------000000000----,
-------000000000----,
---------0000000----,
---------0000000----,
-----------00000----,
-----------00000----,
--------------------,
--------------------,
--------------------,
--------------------}
\caption{set $I(\hat M)\in S_{\check\rho}$}
\end{subfigure}
\hfill
\begin{subfigure}[t]{0.3\textwidth}
\centering
\showgoboard{
--------------------,
--------------------,
--0-0-0-0-0-0-0-0---,
--------------------,
--0-0-0-0-0-0-0-0---,
--------------------,
--0-0-0-0-0-0-0-0---,
--------------------,
--0-0-0-0-0-0-0-0---,
--------------------,
--0-0-0-0-0-0-0-0---,
--------------------,
----0-0-0-0-0-0-0---,
--------------------,
------0-0-0-0-0-0---,
--------------------,
------0-0-0-0-0-0---,
--------------------,
--------0-0-0-0-0---,
--------------------,
----------0-0-0-0---,
--------------------,
--------------------,
--------------------}
\caption{set $R(I(\hat M))\in S_{\hat\rho}$}
\end{subfigure}
\caption{Small computational example with
layout as in commutative diagram in Figure \ref{fig:overview}.
Pair (b) computed from (a) by Algorithm \ref{Alg:refine},
pair (c) computed from (b) by Algorithm \ref{Alg:coarsen}.
Also illustrates stability properties \eqref{new:boundary:in:old:boundary}
and \eqref{I:new:boundary:in:old:boundary}, as well as the effect 
from equation \eqref{I:R:almost:id}.\label{fig:computational:example}}
\end{sidewaysfigure*}

\subsection*{Acknowledgements}

This partly is partly supported by an Australian Government Research Training
Program (RTP) Scholarship.
The TikZ code from \cite{Gaborit} was very helpful for generating the graphics in Figures
\ref{fig:anatomy}, \ref{fig:axiom:failure}, \ref{fig:Voronoi} and \ref{fig:computational:example}.

\bibliographystyle{plain}
\bibliography{bdry}

\end{document}